\documentclass[leqno,a4paper,10pt]{amsart}
\usepackage{amsmath,amsthm,amssymb,amsfonts,enumerate,color}
\usepackage[colorlinks, citecolor=blue,pagebackref,hypertexnames=false]{hyperref}
\allowdisplaybreaks[2]
\numberwithin{equation}{section}

\newtheorem{thm}{Theorem}[section]
\newtheorem{prop}[thm]{Proposition}
\newtheorem{lem}[thm]{Lemma}
\newtheorem{remark}[thm]{Remark}

\newcommand{\real}{{\mathbb R}}
\newcommand{\ent}{{\mathbb Z}}

\newcommand{\norm}[1]{\left\Vert#1\right\Vert}
\newcommand{\abs}[1]{\left\vert#1\right\vert}

\newcommand{\M}{{\mathcal M}}
\renewcommand{\P}{{\mathcal P}}
\renewcommand{\a}{\alpha}
\newcommand{\e}{\varepsilon}
\newcommand{\f}{\varphi}
\newcommand{\8}{\infty}

\newcommand{\red}[1]{\textcolor{red}{#1}}

\begin{document}

\title[Differential transforms for fractional Poisson type operator sequence]
{Boundedness of differential transforms for one-sided fractional Poisson-type operator sequence}


\author{Zhang Chao,\ Tao Ma\ and\ Jos\'e L. Torrea}

 \address{School of Statistics and Mathematics \\
             Zhejiang Gongshang University \\
             Hangzhou 310018, P.R. China}
 \email{zaoyangzhangchao@163.com}

\address{School of Mathematics and Statistics, Wuhan University, Wuhan 430072, P.R. China }
\email{tma.math@whu.edu.cn}

\address{Departamento de Matem\'aticas, Facultad de Ciencias, Universidad
Aut\'onoma de Madrid, 28049 Madrid, Spain }
\email{joseluis.torrea@uam.es}

\thanks{The first author was supported by the Natural Science Foundation of Zhejiang Province(Grant No. LY18A010006), the first Class Discipline of Zhejiang - A (Zhejiang Gongshang University- Statistics) and the State Scholarship Fund(No. 201808330097).  The second author was supported by National Natural Science Foundation of China(Grant Nos. 11671308, 11431011) and the independent research project of Wuhan University (Grant No. 2042017kf0209). The third author was supported  by grant MTM2015-66157-C2-1-P (MINECO/FEDER) from Government of Spain.}

\date{}
\maketitle

\begin{abstract}
 Let $\P_{\tau}^\a f$ be given by
\begin{align*}\label{Formu:subordination}
 \P_{\tau}^\a f(t)=\frac 1{4^\alpha \Gamma(\alpha)}
\int_0^{+\infty}\frac{\tau^{2\alpha}e^{-{\tau^2}/(4s)}}{s^{1+\alpha}}f(t-s)ds, \, \tau >0,\,  t \in \real,\,  0 < \alpha <1.
\end{align*}
It is known that the function $U^\alpha(t,\tau) = \P^\alpha_\tau f(t)$ is a classical solution  to the extension problem
$$
-D_{\rm left}U^\alpha+\frac{1-2\a}{\tau}\,U^\alpha_\tau+U^\alpha_{\tau\tau}=0, \quad \hbox{in}~ \real\times(0,\infty)$$
and
$$\lim_{\tau\to0^+}\P_\tau^\alpha f(t)=f(t),\quad   a.e.\ \hbox{and in}~L^p(\real, w)\hbox{-norm},  w \in A_p^-.$$
In this paper, we analyze the convergence speed of a series related with $\P_{\tau}^\a f$ by discussing
 the behavior of the family of operators
 \begin{equation*}
 T_N^\a f(t)=\sum_{j=N_1}^{N_2} v_j(\P_{a_{j+1}}^\a f(t)-\P_{a_j}^\a f(t)),\quad  ~N=(N_1,N_2)\in \ent^2\quad  \hbox{with} \quad N_1<N_2,
\end{equation*}
where $\{v_j\}_{j\in \mathbb Z}$ is a bounded number sequence, and  $\{a_j\}_{j\in \ent}$ is  a $\rho$-lacunary sequence of positive numbers, that is,
 $1<\rho \leq a_{j+1}/a_j,  \text{for all}\ j\in \mathbb Z.$
 We shall show the boundedness of the maximal operator
  \begin{equation*}\label{Formu:MaxSquareFun}
 T^*f(t)=\sup_N \abs{T_N^\alpha f(t)}, \quad t\in\real,
\end{equation*}
in the one-sided weighted Lebesgue spaces $L^p(\mathbb{R},\omega)(\omega \in A_p^-$), $1< p < \infty$. As a consequence we infer the existence of the limit, in norm and almost everywhere,   of the family $ T_N^\a f$ for functions in $L^p(\mathbb{R},\omega)$. Results for $L^1(\mathbb{R},\omega)(\omega \in A_1^-)$, $L^\infty(\real)$ and $BMO(\real)$ are also obtained.

It is also shown that the local size of  $T^*f$, for functions $f$ having local support, is the same with the order  of a singular integral. Moreover, if $\{v_j\}_{j\in \mathbb Z}\in \ell^p(\mathbb Z)$, we get an intermediate size between the local size of singular integrals and Hardy-Littlewood maximal operator.

 \end{abstract}

\bigskip
\section{Introduction}

Let $\P_{\tau}^\a f$ be given by
\begin{align}\label{Formu:subordination}
 \P_{\tau}^\a f(t)=\frac 1{4^\alpha \Gamma(\alpha)}
\int_0^{+\infty}\frac{\tau^{2\alpha}e^{-{\tau^2}/(4s)}}{s^{1+\alpha}}f(t-s)ds, \, \tau >0,\,  t \in \real,\,  0 < \alpha <1.
\end{align}
This is a fractional Poisson-type operator on the line, which can be found in \cite{Bernardis}.
It is known that the Poisson-type operator appeared when solving the extension problem, see
\cite{CaffarelliSil, StingaTorreaExten, StingaTorreaRegu}.
In \cite{Bernardis}, the authors showed that $\P_{\tau}^\a$ is a classical solution to a version of extension problem for the given initial data $f$ in a weighted space $L^p(w)$, where $w$ satisfies the one-sided $A_p$ condition. Moreover, in this extension problem, they proved that the fractional derivatives on the line are Dirichlet-to-Neumann operators. Precisely,
 it is shown that  for functions $f \in L^p(\mathbb{R}, w), w \in A_p^-, 1<p< \infty$, the function $U^\alpha(t,\tau) = \P^\alpha_\tau f(t)$ is a classical solution  to the extension problem
$$\begin{cases}
\displaystyle -D_{\rm left}U^\alpha+\frac{1-2\a}{\tau}\,U^\alpha_\tau+U^\alpha_{\tau\tau}=0,& \hbox{in}~\real\times(0,\infty), \\
\displaystyle \lim_{\tau\to0^+}\P^\alpha_\tau f(t)=f(t), & a.e.\ \hbox{and in}~L^p(\real, w)\hbox{-norm.}
\end{cases}$$
Moreover, for $\displaystyle c_\a :=\frac{4^{\a-1/2}\Gamma(\a)}{\Gamma(1-\a)}>0$,
$$-c_\a\lim_{\tau\to0^+}\tau^{1-2\a}U^\alpha_\tau(t,\tau)=(D_{\rm left})^\alpha f(t),\quad\hbox{in the distributional sense}.$$
In the above formulas, $$\displaystyle D_{\rm left} f(t) = \lim_{s\to  0^-} \frac{f(t)-f(t-s)}{s}\quad \hbox{ and}\quad
\displaystyle(D_{\rm left})^\alpha f(t) = \frac1{\Gamma(-\alpha)} \int_0^\infty \frac{f(t-s)-f(t)}{s^{\alpha+1}} ds.$$
By $A_p^-$ we denote the class of lateral weights introduced by E.Sawyer  \cite{Sawyer}, see  (\ref{A1+})  and (\ref{Ap+}).

The purpose of this note is to give some extra information about the convergence of the family $\displaystyle \{\P_{\tau}^\a f\}_{\tau>0}$. In order to do this, we shall discuss the behavior of the series
\begin{equation*}\label{diferential}
 \sum_{j\in \ent} v_j(\P_{a_{j+1}}^\a f(t)-\P_{a_j}^\a f(t)),
\end{equation*}
where $\{v_j\}_{j\in \mathbb Z}$ is a sequence of bounded numbers  and  $\{a_j\}_{j\in \ent}$ is  a $\rho$-lacunary sequence of positive numbers, that is,
 $1<\rho \leq a_{j+1}/a_j,  \text{for all}\ j\in \mathbb Z.$ This way to analyze convergence of sequences was considered by Jones and Rosemblatt for ergodic averages(see \cite{JR}), and latter by Bernardis et al. for differential transforms(see \cite{BLMMDT}).

For each $N\in \ent^2,~N=(N_1,N_2)$ with $N_1<N_2$, we define the sum
\begin{equation}\label{Equ:FinSquareFun}
 T_N^\a f(t)=\sum_{j=N_1}^{N_2} v_j(\P_{a_{j+1}}^\a f(t)-\P_{a_j}^\a f(t)).
\end{equation}
We shall consider the   maximal operator
\begin{align}\label{Formu:MaxSquareFun}
\nonumber & T_\alpha^*f(t)=\sup_N \abs{T_N^\alpha f(t)}, \quad t\in\real. \\
\hbox {Along the paper,}&\hbox{ we shall denote}   T^*  \hbox{ to be }  T_\alpha^*   \hbox{ for simply}.
\end{align}
The supremum are taken over all $N=(N_1,N_2)\in \ent^2$ with $N_1< N_2$.

In order to prove the  results, we  shall use the   vector-valued Calder\'on-Zygmund theory  in an essential way.  In  the proof of   the maximal operator $T^*$, we shall use a kind of Cotlar's lemma that in some sense is parallel to the classical Cotlar's inequality   used to control the maximal operator of the truncations in the Calder\'on-Zygmund theory. Looking  to the first set of our results, the reader could have the impression that  the operator $T^*$  is of the same size of the maximal operator $\mathcal{M}^-.$  In this line of thought we present a series of results contained in  Theorem \ref{Thm:infny} and Theorem \ref{Thm:GrothLinfinity} in which it is shown that the size of $T^*$ acting over functions of compact support is in fact of the order of a singular integral. At this point we want to observe the analogy of our operators with   martingale transforms.
  On the other hand  if we consider the sequence of Rademacher functions  $\{r_j\}_{j\in \mathbb Z}$, by Kintchine's inequality we have
\begin{eqnarray*}\left\|\Big(\sum_{j\in \ent}  |\P_{a_{j+1}}^\a f(\cdot)-\P_{a_j}^\a f(\cdot)|^2 \Big)^{1/2}\right\|_{L^p(\mathbb{R})}  \le K_p\left\|\,\, \left\| \sum_{j\in \ent} r_j(\cdot)(\P_{a_{j+1}}^\a f(\cdot)-\P_{a_j}^\a f(\cdot))\right\|_{L^p(\Omega)}\,\,\right\|_{L^p(\mathbb{R})}.
\end{eqnarray*}
In other words, as a by product of our results we get the boundedness of the operator $$ \Big(\sum_{j\in \ent}  |\P_{a_{j+1}}^\a f(\cdot)-\P_{a_j}^\a f(\cdot)|^2 \Big)^{1/2}$$ in the same spaces that we get for operator $T^*$.
Finally, in Theorem
\ref{Thm:GrothLinfinity} it is also shown that if we assume the sequence $\{v_j\}_{j\in \mathbb Z}\in \ell^p(\mathbb Z)$, then the local behavior of $T^*$ is approaching to the maximal operator as $p\rightarrow 1^+.$
Now we present our main results.
\begin{thm}\label{thm:Maxi} Let $0<\alpha <1$, $\{v_j\}_{j\in \mathbb Z}$ a sequence of bounded numbers and  $\{a_j\}_{j\in \ent}$   a $\rho$-lacunary sequence of positive numbers.  Let  $T^*$  be  defined in \eqref{Formu:MaxSquareFun}.
\begin{enumerate}[(a)]
    \item For any $1<p<\infty$ and $\omega\in A_p^-$,  there exists a constant $C$ depending
    on $p, \rho, \alpha, \omega$ and $\norm{v}_{l^\infty(\mathbb Z)}$ such that
 $$\norm{T^*f}_{L^p(\mathbb R, \omega)}\leq C\norm{f}_{L^p(\mathbb R, \omega)},$$
 for all functions $f\in L^p(\mathbb R, \omega).$
    \item For any  $\omega\in A_1^-$, there exists a constant $C$ depending  on $\rho, \alpha, \omega $ and $\norm{v}_{l^\infty(\mathbb Z)}$ such that
 $$\omega\left({\{-\infty<t<+\infty:\abs{T^*f(t)}>\lambda\}}\right) \le C\frac{1}{\lambda}\norm{f}_{L^1(\mathbb R, \omega)},  \quad \lambda >0,$$
for all functions $f\in L^1(\mathbb R, \omega).$
    \item Given $f\in L^\infty(\real),$ then either $T^* f(t) =\infty$ for all $t\in \mathbb R$, or $T^* f(t) < \infty$ for $a. e.$  $t\in \mathbb R$. And in this later case, there exists a constant $C$ depending on $\rho$, $\alpha$ and  $\norm{v}_{l^\infty(\mathbb Z)}$ such that
        \begin{equation*}\norm{T^*f}_{BMO(\mathbb R)}\leq C\norm{f}_{L^\infty(\mathbb R)}.
        \end{equation*}
\item Given $f\in BMO(\real),$ then either $T^* f(t) =\infty$ for all $t\in \mathbb R$, or $T^* f(t) < \infty$ for $a. e.$  $t\in \mathbb R$.  And in this later case,  there exists a constant $C$ depending  on $\rho$, $\alpha$ and  $\norm{v}_{l^\infty(\mathbb Z)}$ such that
\begin{equation}\label{sharp}\norm{T^*f}_{ BMO(\mathbb R)}\leq C\norm{f}_{BMO(\mathbb R)}.
\end{equation}
\end{enumerate}
\end{thm}

We have denoted by  $L^p(\mathbb{R},\omega),  1 \le p < \infty, $ the Lebesgue space of measurable functions satisfying
$$\int_{\mathbb{R}} |f(t)|^p \omega(t) dt < \infty,$$ and
$L^\infty(\mathbb{R})$ the space of measurable functions such that
$\displaystyle \mathop{\hbox{ess sup}}_{t\in\mathbb{R}} |f(t)| < \infty$. Both of them are with the obvious norms.
Also, we define $BMO(\mathbb{R})$ as the space of measurable functions such that  for any interval $B$,
$$ \frac1{|B|} \int_B  \Big| f(t)-f_B\Big| dt \le C <\infty,$$ and
$\displaystyle \|f \|_{BMO(\real)} =  \sup_B \frac1{|B|} \int_B  \Big| f(t)- f_B\Big| dt$, where $\displaystyle f_B=\frac 1 {|B|}\int_B f(t)dt$.   For more details, see \cite{Duo}.

The proof of the last theorem contains three steps:
\begin{itemize}
\item[(A)] We prove  the following uniform boundedness of the family of operators $T_N^\alpha$: from $L^p(\mathbb R, \omega)$ into $L^p(\mathbb R, \omega)$, $1<p< \infty$, from $L^1(\mathbb R, \omega)$ into weak-$L^1(\mathbb R, \omega)$, from $L^\infty(\real)$ into  ${BMO(\real)}$, and from ${BMO(\real)}$  into ${BMO(\real)}$, see Theorem \ref{Thm:PoissonLpOne}.
\item[(B)] The following pointwise Cotlar's type inequality
\begin{equation*}
\sup_{-M\le N_1<N_2\le M} |T^\alpha_{ N_1,N_2}f(t)|\le C\left\{\M^- (T_{-M, M}^\alpha f)(t)+\M^-_q f(t)\right\},
\end{equation*}
see Theorem \ref{Thm:Maximalcontrol}.
\item[(C)] The boundedness of $\M^-$ and the uniform boundedness of $T_{-M, M}^\alpha $ in $L^p(\mathbb R, \omega)$ show the boundedness of the maximal operator $T^*$ in $L^p(\mathbb R, \omega)$. The use of the vector-valued Calder\'on-Zygmund theory allows us to get all of the statements in Theorem \ref{thm:Maxi}.
\end{itemize}

The last theorem has the following consequence.

\begin{thm}\label{Thm:ae}\begin{enumerate}[(a)]
    \item If $1<p<\infty$ and $\omega\in A_p^-$, then $T^\alpha_N f$ converges {\it{a.e.}} and  in $L^p(\mathbb R, \omega)$ norms for all $f\in L^p(\mathbb R, \omega)$ as $N=(N_1,N_2)$ tends to $(-\infty, +\infty).$
   \item If $p=1$  and $\omega\in A_1^-$, then $T^\alpha_N f$ converges {\it{a.e.}} and in measure for all $f\in L^1(\mathbb R, \omega)$ as $N=(N_1,N_2)$ tends to $(-\infty, +\infty).$
\end{enumerate}
\end{thm}

The dichotomy results announced  in Theorem \ref{thm:Maxi}, parts $(c)$ and $(d)$, about $L^\infty(\mathbb R)$ and $BMO(\mathbb R)$   are  motivated, in part,   by the existence of a bounded function $f$ such that $T^*f(t)=\infty$ as the following theorem shows.
 \begin{thm}\label{Thm:infny}
  There exist bounded sequence $\{v_j\}_{j\in \mathbb Z}$, $\rho$-lacunary sequence $\{a_j\}_{j\in \mathbb Z}$  and $f\in L^\infty(\mathbb R)$ such that  $T^* f(t) =\infty$ for all $t\in \mathbb R$.
\end{thm}

This last theorem also says that the operator $T^*$ is essentially bigger than the operator $\P^*f(t) = \sup_\tau \P^{1/2}_\tau f(t)$ which is bounded in $L^p(\real, \omega), 1<p< \infty$, and in $L^\infty(\mathbb{R})$, see  \cite{Bernardis}.

On the other hand,  if $f= \chi_{(0,1)}$ and $\mathcal{H}$ is the Hilbert transform, it is easy to see that   $\displaystyle  \frac1{r} \int_{-r}^0 \mathcal{H}(f)(x)dx\  \sim \log\frac{e}{r}$  as $ r \to 0^+$. In general, this is the growth of a singular integral applied to a bounded function  at the origin. The following theorem shows that the growth of the function  $T^*f$ for  bounded function $f$  at the origin   is of the same order of a singular integral operator.

\begin{thm} \label{Thm:GrothLinfinity}
\begin{enumerate}[(a)]
\item  Let   $\{v_j\}_{j\in \mathbb Z}\in l^p(\mathbb Z)$ for some $1 \le p\le \infty.$ For every $f\in L^\infty(\mathbb{R})$ with support in the unit ball $B=B(0, 1)$, for any ball $B_r\subset B$ with $2r<1$, there exists a constant $C>0$ such that
    $$\frac{1}{|B_r|} \int_{B_r} \abs{T^* f (t)} dt\leq C\left(\log \frac{2}{r}\right)^{1/p'}\norm{v}_{l^p(\mathbb Z)}\|f\|_{L^\infty(\mathbb R)}.$$
\item When $1< p<\infty$, for any $\varepsilon>0$, there exist a $\rho$-lacunary sequence  $\{a_j\}_{j\in \mathbb Z}$,  a sequence $\{v_j\}_{j\in \mathbb Z}\in \ell^p(\mathbb Z)$ and a function  $f\in L^\infty(\mathbb{R})$ with support in the unit ball $B=B(0, 1),$ satisfying the following statement: for any ball $B_r\subset B$ with $2r<1$, there exists a constant $C>0$ such that
    $$\frac{1}{|B_r|} \int_{B_r} \abs{T^* f (t)} dt\geq C\left(\log \frac{2}{r}\right)^{1/(p-\varepsilon)'}\norm{v}_{l^p(\mathbb Z)}\|f\|_{L^\infty(\mathbb R)}.$$
\item When $p=\infty,$ there exist  a $\rho$-lacunary sequence $\{a_j\}_{j\in\mathbb Z}$, a sequence $\{v_j\}_{j\in \mathbb Z}\in l^\infty(\mathbb Z)$ and $f\in L^\infty(\mathbb{R})$ with support in the unit ball $B=B(0, 1)$, satisfying the following statements: for any  ball $B_r\subset B$ with $2r<1$, there exists a constant $C>0$ such that
    $$\frac{1}{|B_r|} \int_{B_r} \abs{T^* f (t)} dt\geq C\left(\log \frac{2}{r}\right)\norm{v}_{l^\infty(\mathbb Z)}\|f\|_{L^\infty(\mathbb R)}.$$
\end{enumerate}
In the statements above,  $\displaystyle p' = \frac{p}{p-1},$ and if $p=1$, $\displaystyle p'=\infty.$
\end{thm}

Some related results  about the local behavior of variation operators can be found in \cite{BCT}. One dimensional results about the variation of some convolutions operators can be found in \cite{MTX}.

\medskip

\medskip

The organization of the paper is as follows.  In Section \ref{Sec:mainproof}, we will get the kernel estimates to see that the kernel $K_N^\alpha$ is a vector-valued Caldr\'on-Zygmund kernel, and then we can get the uniform boundedness of $T_N^\alpha$, i.e.  Theorem \ref{Thm:PoissonLpOne}. And with a Cotlar's inequality, we can get the proof of Theorem \ref{thm:Maxi}  in Section \ref{Sec:max}. In Section \ref{sec:equinfty}, we will give the proof of Theorem \ref{Thm:infny} and  Theorem \ref{Thm:GrothLinfinity}.

\medskip
Throughout this paper, the symbol $C$ in an inequality always denotes a constant which may depend on some indices, but never on the functions $f$ in consideration.

\medskip

\section{Uniform $L^p $ boundedness of the operators $T_N^\alpha$} \label{Sec:mainproof}
We shall need the following lemma.
\begin{lem}\label{ComplexIntegral}
 Let $0<\a<1$. Then for any complex number $z_0$ with $Re z_0 > 0$ and $\displaystyle |\arg z_0 |\leq {\pi}/{4}$, we have
 $$ \int_0^\8  e^{-z_0 u} e^{-\frac{z_0}{u} }\,\frac{du}{u^{\alpha}}= z_0^{1-\a}\int_{0}^{\8}  \frac{e^{-r}e^{- z_0^2/r}}{r^{2-\a}} dr.$$
 \end{lem}
\begin{proof}
Let $ \varphi_0= \arg z_0$. Assume that $ 0\leq \varphi_0 \leq \pi/4$. The case
$-\pi/4 \leq \varphi_0 \leq 0$ is completely analogous. Define the ray in the complex plane
$$\mathrm{ray}_{\f_0}:=\{z=re^{i\varphi_0}:0<r<\8\}.$$
And then let $\mathcal{C}$ denote the sector in the real part of the complex plane, with $0\leq \arg z \leq \f_0 $ but truncated at  $c_\e: |z|=\e$ and $C_R: |z|=\e$. In fact, the boundary of $\mathcal{C}$ consists four parts: $C_\e$, $C_R$, $\mathrm{ray}_{\varphi_0}$ and positive half part of the real line.

Let us consider the complex function
$$F(u)=\frac{e^{-z_0/u}e^{-uz_0}}{u^\a},$$
which is holomorphic function when $u \neq 0$. Thus, by the Cauchy theorem, we have $\displaystyle \int_{\mathcal{C }} F(u)du=0 $. We first calculate
 \begin{align*}
 \abs{ \int_{C_\e}  F(u) du} & =\abs{ \int_0^{\f_0} \frac{e^{-z_0/ (e^{i\theta}\e)}
                           e^{-z_0 \e e^{i\theta}}}{\e^\a e^{i\a \theta} } i\e e^{i\theta} d\theta}
                           =\abs{ \int_0^{\f_0} \frac{e^{-|z_0|e^{i (\f_0 -\theta)}/\e}
                           e^{-|z_0|\e e^{i (\f_0+\theta)}}}{\e^\a e^{i\a \theta} } i \e e^{i\theta} d\theta}.
 \end{align*}
Since $\f_0 <{\pi}/{4}$, $\f_0-\theta<{\pi}/{2}$ and $\f_0+\theta<{\pi}/{2}$.  Hence
\begin{align*}
 \abs{ \int_{C_\e} F(u) du} & \leq \e^{1-\a} \int_0^{\f_0} e^{-|z_0| \cos(\f_0 -\theta)/\e}
                           e^{-|z_0|\e\cos(\f_0 +\theta)} d \theta\to 0,
\end{align*}
as $\e\to 0$. Similarly, along the curve $C_R$, we have
\begin{align*}
 \abs{ \int_{C_\e} F(u) du} & \leq \int_0^{\f_0} e^{-|z_0| \cos(\f_0 -\theta)/R}
                           e^{-|z_0|R\cos(\f_0 +\theta)}R^{1-\a}d\theta.
\end{align*}
If $\f_0<{\pi}/{4}$,
\begin{align*}
 \abs{\int_{C_R}  F(u) du} &\leq e^{-C_{z_0}  R^2}R^{1-\a} \int_0^{\f_0} e^{-\cos(\f_0-\theta)}d \theta\to 0,
\end{align*}
as $R\to \8$. But for the case $\f_0={\pi}/{4}$,  $\f_0+\theta$ can be ${\pi}/{2}$, then we can not take the limit as above. However, we have
\begin{align*}
 \abs{\int_{C_R} F(u) du} & \leq \int_0^{\frac{\pi}{4}} e^{-|z_0| \cos(\frac{\pi}{4} -\theta)/R}
                           e^{-|z_0|R\cos(\frac{\pi}{4} +\theta)}R^{1-\a}d\theta\\
                           & \leq \int_0^{\frac{\pi}{4}} e^{-|z_0|R\cos(\frac{\pi}{4} +\theta)}R^{1-\a}d\theta\leq \int_0^{\frac{\pi}{4}} e^{-|z_0|R\sin(\frac{\pi}{4}-\theta)}R^{1-\a}d\theta\\
                           & \leq \int_0^{\frac{\pi}{4}} e^{-|z_0|R\sin\omega}R^{1-\a}d\omega\leq \int_0^{\frac{\pi}{4}} e^{-|z_0|\frac{2}{\pi}R \omega}R^{1-\a}d\omega,
\end{align*}
where we have changed variable $\displaystyle \omega= {\pi}/{4}-\theta$ and used the inequality $\displaystyle{2\omega}/{\pi} \leq \sin \omega$. Thus we have
$$ \abs{ \int_{C_R} F(u) du} \leq \frac{\pi}{2|z_0|} R^{-\a}\int_0^{\8} e^{-u} d u\leq C R^{-\a}\to 0, \quad R\to \8.$$
Therefore, we conclude that $\displaystyle \abs{ \int_{C_R} F(u) du}=0$ for $ |\arg z |\leq {\pi}/{4}$.

\vspace{0.5em}

At last, by the Cauchy theorem, we then get
\begin{align*}
 \int_0^\8 F(u) du &= \int_{\mathrm{ray}_{\f_0}} F(u) du.
\end{align*}
Taking $u= sz_0$, we have
\begin{align*}
 \int_0^\8 F(u) du = \int_{\mathrm{Ray}_{\f_0}} \frac{e^{-1/s}e^{-s z_0^2}}{s^\a z_0^\a} z_0ds = z_0^{1-\a}\int_{0}^{\8}  \frac{e^{-r}e^{- z_0^2/r}}{r^{2-\a}} dr.
\end{align*}
Then this lemma is completely proved.
\end{proof}

\begin{remark}\label{Rem:ParaPoisson}
Notice that the  integral
$$\frac{y^{2s}}{4^s\Gamma(s)}\int_0^\infty e^{-y^2/(4\tau)}e^{-\tau(i\rho+\lambda)}\,\frac{d\tau}{\tau^{1+s}},
\quad\rho\in\real,~\lambda\ge0,~0<s<1.$$
is absolutely convergent.
\end{remark}

\subsection{Uniform $L^2$-boundedness}
\

It is known that, see \cite{Bernardis}, the Fourier transform of $\P^\alpha_\tau f$ is
\begin{equation*}\label{Fourier}
\widehat{\P^\alpha_\tau f}(\rho) = \frac1{\Gamma(\alpha)} \int_0^\infty e^{-r} e^{-i\rho\tau^2/4r} \hat{f}(\rho)\frac{dr}{r^{1-\alpha}}.
\end{equation*}
By $\widehat{f}(\rho)$ we denote the Fourier transform of the function $f$, that is,
$$ \widehat{f}(\rho) = \frac{1}{(2\pi)^{1/2}}\int_{\mathbb{R}}f(x)e^{-ix\rho} \,dx,\quad \rho\in\mathbb{R}.$$

\begin{thm}\label{Thm:L2EstimateOne}
There is a constant $C$, depending  on $\alpha$ and $\norm{v}_{l^\infty(\mathbb Z)}$, such that
 $$\sup_N \|T_N^\a f \|_{L^2(\real)}\leq C \|f \|_{L^2(\real)}.$$
\end{thm}

\begin{proof}
Let $f\in L^2(\real)$. Using  the Plancherel theorem, we have
\begin{align*}
 \norm{T_N^\a f }_{L^2(\real)} & = \norm{\sum_{j=N_1}^{N_2} v_j(\P_{a_{j+1}}^\a f -\P_{a_j}^\a f)}_{L^2(\real)} \leq  C\norm{v}_{l^\infty(\mathbb Z)}
   \norm{\sum_{j=-\infty}^{\infty} \int_{a_j}^{a_{j+1}} \abs{\partial_\tau \widehat{\P_{\tau}^\a f }} d\tau}_{L^2(\real)}.
\end{align*}
Observe that,
\begin{align*}
\partial_\tau  \widehat{\P_{\tau}^\a f }(\rho) &
     = C \partial_\tau \int_0^\8 e^{-r} e^{-\frac{\tau^2}{4r}(i\rho)}\widehat{f} (\rho)
       \,\frac{dr}{r^{1-\alpha}}
    = C  \int_0^\8 e^{-r}\tau (i\rho) e^{-\frac{\tau^2}{4r}(i\rho)}\widehat{f} (\rho)
       \,\frac{dr}{r^{2-\alpha}}.
\end{align*}
Note that the Fourier transform above is well defined, see Remark \ref{Rem:ParaPoisson}. Then we deduce that
\begin{align*}
 \norm{T_N^\a f }_{L^2(\real)} & \leq  C
   \norm{  \widehat{f} (\rho) \int_{0}^\8 \abs{\int_0^\8 e^{-r}\tau (i\rho) e^{-\frac{\tau^2}{4r}(i\rho)}\,\frac{dr}{r^{2-\alpha}}  } d\tau}_{L^2(\real)}.
\end{align*}
Changing variable $z_0= \tau \sqrt{i\rho}$, by using Lemma \ref{ComplexIntegral} , we have
\begin{align*}
\Big| \int_{0}^\8\abs{ \int_0^\8 e^{-r}\tau (i\rho) e^{-\frac{\tau^2}{4r}(i\rho)}\,\frac{dr}{r^{2-\alpha}}} d\tau
& = \int_{0}^\8 \abs{\int_0^\8  e^{-r}z_0 e^{-\frac{z_0^2}{4r} }
       \,\frac{dr}{r^{2-\alpha}}   }dz_0 \\
    &= 2^{1-\alpha}\int_{0}^\8 \abs{z_0^\a \int_0^\8  e^{-\frac{z_0}{2u}}e^{-\frac{z_0}{2}u} \frac{du}{u^\a}  }dz_0.
\end{align*}
 Since $\abs{\arg z_0}= {\pi}/{4}$, we have $|e^{-z_0/(2u)}| \leq e^{-c|z_0|/u} $ and $|e^{-z_0 u/2}| \leq e^{-c|z_0| u}$, where $c={ \sqrt{2}/ {4}}$. Then
\begin{align*}
&\abs{\int_{0}^\8 z_0^\a \int_0^\8  e^{-z_0/u}e^{-z_0 u} \frac{du}{u^\a}  dz_0}  \leq
  \int_{0}^\8 |z_0|^\a \int_0^\8  e^{-c|z_0|/u}e^{-c|z_0| u} \frac{du}{u^\a} dz_0\\
&\leq \int_{0}^\8 |z_0|^{2\a-1}\int_0^\8  e^{-c|z_0|^2/v}e^{-c v} \frac{d v}{v^\a} dz_0
  = \int_{0}^\8 |\sqrt{i\rho}|^{2\a}\,\tau^{2\a-1}\int_0^\8  e^{-c(|\sqrt{i\rho}|\tau)^2/v}e^{-c v} \frac{d v}{v^\a} d\tau\\
&    = \int_{0}^\8\int_0^\8 (|\sqrt{i\rho}|\tau)^{2\a-1} e^{-c(m\tau)^2/v} d(|\sqrt{i\rho}|\tau) e^{-c v} \frac{d v}{v^\a}\\
&   = \int_{0}^\8\int_0^\8  \tau^{2\a-1} e^{-c\tau^2/v} d\tau e^{-c v} \frac{d v}{v^\a}
\leq C \int_{0}^\8   e^{-c v} d v \leq C.
\end{align*}
Then the proof of the theorem is complete.
\end{proof}

\subsection{ Uniform $L^p$-boundedness}

\

Let us come back to the definition of the operators $T_N^\alpha$, see (\ref{Equ:FinSquareFun}). By using the formula  (\ref{Formu:subordination}), we have

\begin{align*}
T_N^{\a} f(t) &=\sum_{j=N_1}^{N_2} v_j(\P_{a_{j+1}}^\a f(t)-\P_{a_j}^\a f(t)) \\
     &= \frac{1}{4^\alpha\Gamma(\alpha)}  \sum_{j=N_1}^{N_2}v_j \int_0^{+\infty} \frac{ a_{j+1}^{2\a}e^{-a_{j+1}^2/(4 s)}-a_j^{2\a} e^{-a_j^2/(4 s)}        }{s^{1+\a}} f(t-s)~ds\\
     &=\int_0^{+\infty} K_N^\alpha(s) f(t-s)~ds=\int_{-\infty}^{t} K_N^\alpha(t-s) f(s)~ds,
\end{align*}
where
\begin{equation}\label{kernel1dim}K_N^\a(s)=\frac{1}{4^\alpha\Gamma(\alpha)} \sum_{j=N_1}^{N_2}v_j \frac{ a_{j+1}^{2\a}e^{-a_{j+1}^2/(4 s)}-a_j^{2\a} e^{-a_j^2/(4 s)}}{s^{1+\a}}.
\end{equation}

The kernel $K_N^\alpha(s)$ is supported in $(0, +\infty).$ Our study of $T^\alpha_N$ will be  related to the one-sided Calder\'on-Zygmund operators. In particular,
we shall look for Lebesgue estimates with absolute continuous measures  $w(x)dx$, where $w$ is a weight in any of the classes $A_p^{\pm}$ defined by E. Sawyer, see \cite{Sawyer}. This classes were introduced in relation with the boundedness of
the one-sided Hardy-Littlewood maximal operator $\M^-$ defined by
 \begin{equation*}\label{equ:oneMax}
\M^- f(t)=\sup_{\varepsilon>0}\frac{1}{\varepsilon}\int_{-\varepsilon}^0\abs{f(t+s)}ds.
\end{equation*}  We recall the  results that we shall use related with  weights for $\M^-:$
\begin{enumerate}[(1)]

\item  The operator $\M^-$ is of weak type $(1,1)$ with respect to the measure $\omega(t)dt$ if and only if $\omega\in A_1^-$, i.e., there exists $C$ such that \begin{equation}\label{A1+} \M^+\omega\le C\omega\ \  a.e., \end{equation}  where $\M^+$ is the right-sided Hardy Littlewood maximal operator defined as $$\displaystyle \M^+f(t)=\sup_{\varepsilon>0}{1\over \varepsilon}\int_0^\varepsilon \abs{f(t+s)}ds.$$
\item The operator $\M^-$ is bounded in $L^p(\omega)$, $1<p<\infty,$ if and only if $\omega\in A_p^-,$ i.e., if there exists $C$ such that for any three points $a<b<c$
    \begin{equation}\label{Ap+}
     \left(\int_a^b \omega^{1-p'}\right)^{1\over p'}\left(\int_b^c \omega\right)^{1\over p}\le C(c-a),
    \end{equation}
    where $\displaystyle {1\over p}+{1\over p'}=1.$
\end{enumerate}
For more details about the one-sided weights, see \cite{AFM, BLMMDT, Sawyer}.

\begin{thm}\label{Thm:KernelEstOne} Let $K_N^\alpha$ be the kernel defined in (\ref{kernel1dim}).
For any $s\neq 0,$ there exists  constant $C$ depending  on $\alpha$ and $\norm{v}_{l^\infty(\mathbb Z)}$(but not on $N$)  such that
\begin{enumerate}[\indent i)]
  \item $\displaystyle |K_N^\a(s)|\leq \frac{C}{s}$,
    \item  $\displaystyle |\partial_s K_N^\a(s)|\leq \frac{C}{s^2}$.
\end{enumerate}
\end{thm}

The proof of Theorem \ref{Thm:KernelEstOne} involves an estimate we will repeat several times, so we formulate it in the following remark.

\begin{remark}\label{Lem:KernelLpEst}
Along the paper, we shall use frequently the estimate  $x^A e^{-x/B}  \le Ce^{-x/B'}$ with $x, A,B,B',  C >0.$
\end{remark}

\begin{proof}[Proof of Theorem \ref{Thm:KernelEstOne}] For $i)$, we have

\begin{align*}
  |K_N^\a (s)| &\leq C\sum_{j=-\8}^{\8} \left|\frac{a_{j+1}^{2\a} e^{-a_{j+1}^2/(4s)}-a_j^{2\a} e^{-a_j^2/(4s)}}{s^{1+\a}}  \right|
  = C\sum_{j=-\8}^{\8} \left| a_{j+1}^{2\a} e^{-a_{j+1}^2/(4s)}-a_j^{2\a} e^{-a_j^2/(4s)}\right| \frac{1}{s^{1+\a}}.
\end{align*}
Observe that, by Remark \ref{Lem:KernelLpEst},
\begin{align*}
&\sum_{j=-\8}^{\8} \left| a_{j+1}^{2\alpha}e^{-a_{j+1}^2/(4s)}-a_j^{2\alpha}e^{-a_j^2/(4s)}\right|
        = \sum_{j=-\8}^{\8} \left|\int_{a_j}^{a_{j+1}}\partial_u \left(u^{2\a} e^{-u^2/(4s)}\right) du\right|\nonumber\\
       & \leq \int_{0}^{\infty}\left|( 2\a u^{2\a-1}-\frac{u^{2\a+1}}{2s})e^{-u^2/(4s)} \right| du     \leq C\int_{0}^{\infty}\left|(u^{2\a-1}+\frac{u^{2\a+1}}{2s})e^{-u^2/(4s)} \right| du\nonumber\\
       &\leq C\sqrt{s}\Big(\int_{0}^{\infty} (\sqrt{s} )^{2\a-1}  \left(\frac{u}{\sqrt{s}} \right)^{2\a-1} e^{-\frac14 \left(u/\sqrt{s}\right)^2} d\frac{u}{\sqrt{s}}\\
       &\quad  + s^{\a-1/2}\int_{0}^{\infty}\left(\frac{u}{\sqrt{s}}\right)^{2\a+1}  e^{-\frac14\left(u/\sqrt{s} \right)^2} d\frac{u}{\sqrt{s}}\Big)\nonumber\\
       &\leq C s^{\a}.\nonumber
\end{align*}
Then
$ \displaystyle
  |K_N^\a (s)| \leq \frac{C}{s}$.  This proves {\it i)}.

For {\it ii)}, we can write
\begin{align*}
 K_N^{\a}(s) &= C\sum_{j=N_1}^{N_2} v_j\left( a_{j+1}^{2\alpha}e^{-a_{j+1}^2/(4s)}-a_j^{2\alpha}e^{-a_j^2/(4s)}\right) \\ &
               = C\sum_{j=N_1}^{N_2} \frac{1}{s^{1+\a} }
                   v_j\int_{a_j}^{a_{j+1}} \left(2\a u^{2\a-1}-\frac{u^{2\a+1}}{2s}\right) e^{-\frac{u^2}{4s}}du.
\end{align*}
The partial derivative $\partial_{s} K_N^\a(s)$ consists two parts. The first part is
\begin{align*}
I &= C\sum_{j=N_1}^{N_2} \frac{1}{s^{1+\a} } v_j \int_{a_j}^{a_{j+1}}\left(\frac{u^{2\a+1}}{2s^2}
      + \left(2\a u^{2\a-1}-\frac{u^{2\a+1}}{2s}\right)\frac{u^2}{4s^2} \right) e^{-\frac{u^2}{4s}}du\\
  &= C\sum_{j=N_1}^{N_2} \frac{1}{s^{1+\a} } v_j \int_{a_j}^{a_{j+1}}\left(\frac{(\a+1)u^{2\a+1}}{2 s^2}
      -\frac{u^{2\a+3}}{8s^3}\right) e^{-\frac{u^2}{4s}}du.
\end{align*}
And the second part is
\begin{align*}
II &= C\sum_{j=N_1}^{N_2} \partial_s\left(\frac{1}{s^{1+\a}} \right)v_j \int_{a_j}^{a_{j+1}} \left(2\a u^{2\a-1}
       -\frac{u^{2\a+1}}{2s}\right) e^{-\frac{u^2}{4s}}du\\
   &= C\sum_{j=N_1}^{N_2}  \left(-\frac{1+\a}{s^{2+\a}}\right)
       v_j\int_{a_j}^{a_{j+1}} \left(2\a u^{2\a-1} -\frac{u^{2\a+1}}{2s}\right) e^{-\frac{u^2}{4s}}du.
\end{align*}
Then by using Remark  \ref{Lem:KernelLpEst} again, we have
\begin{align*}
|I| &\leq C \frac{1}{s^{1+\a} } \int_{0}^{\8} \abs{\frac{(\a+1)u^{2\a+1}}{2 s^2}
      -\frac{u^{2\a+3}}{8s^3}} e^{-\frac{u^2}{4s}}du
 \leq C \frac{1}{s^{1+\a} } s^{\a-1}\leq  \frac{C}{s^2},
\end{align*}
and
\begin{align*}
|II|  &\leq C \Big| \frac{1+\a}{s^{2+\a}} \Big|
       \int_{0}^{\8} \Big| (2\a u^{2\a-1} -\frac{u^{2\a+1}}{2s}) e^{-\frac{u^2}{4s}}du \Big|
      \leq C \frac{s^\a}{s^{2+\a} }\leq \frac{C}{s^2}.
\end{align*}
Combining the estimates $I$ and $II$, we have
\begin{align*}
|\partial_{s} K_N^\a (s) | \leq \frac{C}{s^2}.
\end{align*}
All the estimates above are true uniform for $N$. The proof of the Theorem \ref{Thm:KernelEstOne} is complete.
\end{proof}

  From  Theorems  \ref{Thm:L2EstimateOne}, \ref{Thm:KernelEstOne}, and standard Calder\'on-Zygmund theory, we can get the uniform estimate in $L^p(\real, w)$ ($1 < p< \infty, \ w\in A_p$) of the operators $T_N^\alpha$.
 Here, $A_p$ denotes the classical Muckenhoupt $A_p$ weights, see \cite{RubioRuTo}.
  However, to the one-side nature of the kernel, we can apply Theorem 2.1 in \cite{AFM}  to get  the uniform boundedness in $L^p(\real, w) $  of the operators $T_N^\alpha$ with   $w \in A_p^-$ in the following.

\begin{thm}\label{Thm:PoissonLpOne}
Let $T_N^\a$ be the family of operators defined   in \eqref{Equ:FinSquareFun}, we have the following statements.
\begin{enumerate}[(a)]
    \item For any $1<p<\infty$ and $\omega\in A^-_p$,  there exists a constant $C$ depending  on $p, \alpha$,  $\norm{v}_{l^\infty(\mathbb Z)}$  and $\omega$(not on $N$) such that
 $$\norm{T_N^\alpha f}_{L^p(\mathbb R, \omega)}\leq C\norm{f}_{L^p(\mathbb R, \omega)},$$
 for all functions $f\in L^p(\real, \omega).$
    \item  For any $\omega\in A^-_1$, there exists a constant $C$ depending  on $\alpha$, $\norm{v}_{l^\infty(\mathbb Z)}$ and $\omega$(not on $N$) such that
 $$\omega\left({\{t\in \real}:\abs{T_N^\alpha f(t)}>\lambda\}\right) \le C\frac{1}{\lambda}\norm{f}_{L^1(\mathbb R, \omega)},\quad   \lambda >0,$$
for all functions $f\in L^1(\real, \omega).$
    \item There exists a constant $C$ depending  on $\alpha$ and $\norm{v}_{l^\infty(\mathbb Z)}$(not on $N$)  such that
$$\norm{T_N^\alpha f}_{BMO(\mathbb R)}\leq C\norm{f}_{L^\infty(\mathbb R)},$$
for all functions $f\in L^\infty(\real).$
\item There exists a constant $C$ depending on $\alpha$ and $\norm{v}_{l^\infty(\mathbb Z)}$(not on $N$) such that
$$\|T_N^\alpha f\|_{BMO(\mathbb R)} \leq C\|f\|_{BMO(\mathbb R)} ,$$
for all functions $f\in BMO(\mathbb R)).$
\end{enumerate}
The constants C appeared above all are independent with $N.$
\end{thm}
As we have said before the proof of $(a)$ and $(b)$ in  the theorem above is  obtained by using Theorem 2.1 in \cite{AFM}. On the other hand the proof of $(c)$ and $(d)$ are standard in the Calder\'on-Zygmund theory and it can be found in \cite{Xu}.

\medskip

\section{Boundedness of the maximal operator $T^*$}\label{Sec:max}

In this section, we will give the proof of Theorem \ref{thm:Maxi} related to the boundedness of the maximal operator $T^*$. The next proposition,  parallel to  Proposition 3.2 in \cite{BLMMDT}, shows that, without lost of generality, we may assume that
\begin{equation}\label{equ:lacunary}
1<\rho \leq {a_{j+1} \over a_j}\leq \rho^2, \quad j\in \mathbb Z.
\end{equation}

\begin{prop}\label{Prop:lacunary}
Given a $\rho$-lacunary sequence $\{a_j\}_{j\in \mathbb Z}$ and a multiplying sequence $\{v_j\}_{j\in \mathbb Z}\in l^\infty(\mathbb Z)$, we can define a $\rho$-lacunary sequence $\{\eta_j\}_{j\in \mathbb Z}$ and $\{\omega_j\}_{j\in \mathbb Z}\in l^\infty(\mathbb Z)$ verifying the following properties:
\begin{enumerate}[(i)]
\item $1<\rho \leq \eta_{j+1}/\eta_j\leq \rho^2,\quad \quad \norm{\omega_j}_{l^\infty(\mathbb Z)}=\norm{v_j}_{l^\infty(\mathbb Z)}$.
\item For all $N=(N_1, N_2)$ there exists $N'=(N_1', N_2')$ with $T_N^\alpha=\tilde{T}_{N'}^\alpha,$
where $\tilde{T}_{N'}^\alpha$ is the operator defined in \eqref{Equ:FinSquareFun} for the new sequences $\{\eta_j\}_{j\in \ent}$ and $\{\omega_j\}_{j\in \ent}.$
\end{enumerate}
\end{prop}

\begin{proof} We  follow closely the ideas in the proof
of Proposition 3.2 in \cite{BLMMDT}. We include it at here for completeness.

Let $\eta_0=a_0$, and let us construct $\eta_j$ for positive $j$ as follows (the argument for
negative $j$ is analogous). If $\rho^2\ge {a_1/ a_0}\ge \rho,$ define $\eta_1=a_1$. In the opposite case where
$a_1/a_0>\rho^2,$  let $\eta_1=\rho a_0.$ It verifies $\rho^2\ge \eta_1/\eta_0=\rho\ge \rho.$ Further, $a_1/\eta_1\ge \rho^2 a_0/\rho a_0=\rho.$  Again, if $a_1/\eta_1\le \rho^2,$ then $\eta_2=a_1.$ If this is not the case, define $\eta_2=\rho^2 a_0\le a_1$.  By
the same calculations as before, $\eta_0, \eta_1, \eta_2$ are part of a lacunary sequence satisfying \eqref{equ:lacunary}.
To continue the sequence, either $\eta_3=a_1$ (if $a_1/ \eta_2\le \rho^2$) or $\eta_2=\rho^3\eta_0$ (if $a_1/\eta_2>\rho^2$). Since $\rho>1,$ this process ends at some $j_0$ such that $\eta_{j_0}=a_1.$ The rest of the elements $\eta_j$
are built in the same way, as the original $a_k$ plus the necessary terms put in between two
consecutive $a_j$ to get \eqref{equ:lacunary}.

Let $J(j)=\{k:a_{j-1}<\eta_k\le a_j\}$, and $\omega_k=v_j$ if $k\in J(j)$. Then
\begin{equation*}
 v_j(\P_{a_{j+1}}^\a f(t)-\P_{a_j}^\a f(t))=\sum_{k\in J(j) }\omega_k(\P_{a_{k+1}}^\a f(t)-\P_{a_k}^\a f(t)).
\end{equation*}
If $M=(M_1, M_2)$ is the number such that $\eta_{M_2+1}=a_{N_2+1}$ and $\eta_{M_1}=a_{N_1}$, then we get
\begin{equation*}
 T_N^\a f(t)=\sum_{j=N_1}^{N_2} v_j(\P_{a_{j+1}}^\a f(t)-\P_{a_j}^\a f(t))=\sum_{k=M_1}^{M_2} \omega_k(\P_{\eta_{k+1}}^\a f(t)-\P_{\eta_k}^\a f(t))= \tilde{T}_M^\a f(t),
\end{equation*}
where $\tilde{T}_M^\a$  is the operator defined in \eqref{Equ:FinSquareFun} related with sequences $\{\eta_k\}_{k\in \mathbb Z}$, $\{\omega_k\}_{k\in \mathbb Z}$, $\alpha$ and $M=(M_1, M_2)$.
\end{proof}

It follows from this proposition that it is enough to prove all the results of this article
in the case of a $\rho$-lacunary sequence satisfying \eqref{equ:lacunary}. For this reason, in the rest of the article
we assume that  $\{a_j\}_{j\in \mathbb Z}$  satisfies \eqref{equ:lacunary} without saying it explicitly.

In order to prove Theorem \ref{thm:Maxi}, we need a Cotlar's type inequality to control the operator $T^*$ by some one-sided Hardy-Littlewood maximal operators.

For any $M\in \mathbb Z^+,$ let\\
$$T_M^*f(t)=\sup_{-M\le N_1<N_2\le M}\abs{T_N^\alpha f(t)},\quad -\infty<t<+\infty.$$
\begin{thm}\label{Thm:Maximalcontrol}
For each $q\in (1, +\infty),$ there exists a constant $C$ depending  on $q,$ $\norm{v}_{l^\infty(\mathbb Z)}$, $\alpha$ and $\rho$  such that  for every $M\in \mathbb Z^+$,
\begin{equation*}
T_M^*f(t)\le C\left\{\M^- (T_{-M, M}^\alpha f)(t)+\M^-_q f(t)\right\}, \quad   -\infty<t<+\infty,
\end{equation*}
where
$$\M^-_qf(t)=\sup_{\varepsilon>0} \left(\frac{1}{\varepsilon}\int_{-\varepsilon}^0\abs{f(t+s)}^qds\right)^{1\over q}.$$
\end{thm}

\begin{proof}
Since the operators $T_N^\alpha$ are given by convolutions, they are invariant under translations, and therefore it is enough to prove the theorem for $t=0.$ Observe that, for $N=(N_1, N_2),$
$$T_N^\alpha f(t)=T_{N_1, M}^\alpha f(t)-T_{N_2+1, M}^\alpha f(t),$$
with $-M\le N_1<N_2\le M.$
Then, it suffices to estimate $\abs{T_{m, M}^\alpha f(0)}$ for $\abs{m}\le M$ with constants independent of $m$ and $M.$
Let us split $f$ as
\begin{align*}
f(t)&=f(t)\chi_{(-a^2_{m+1},0]}(t)+f(t)\chi_{(-\infty, -a^2_{m+1}]}(t)+f(t)\chi_{(0,+\infty)}\\&=:f_1(t)+f_2(t)+f_3(t),
\end{align*}
for $-\infty<t<+\infty.$

First, notice that $T^\alpha_{m,M}f_3(0)=0.$ Then,  we have
\begin{align*}
\abs{T_{m,M}^\alpha f(0)}&\le \abs{T_{m,M}^\alpha f_1(0)}+\abs{T_{m,M}^\alpha f_2(0)}\\
&=: I+II.
\end{align*}
For $I$, by the mean value theorem, we have
\begin{align*}
I &=\abs{T_{m,M}^\alpha f_1(0)}\\
  &=C_\alpha \abs{\int_0^{+\infty} \sum_{j=m}^{M}v_j \frac{ a_{j+1}^{2\a}e^{-a_{j+1}^2/(4 s)}-a_j^{2\a} e^{-a_j^2/(4 s)}}{s^{1+\a}}   f_1(-s)ds}\\
  &\le C_\alpha \norm{v}_{l^\infty(\mathbb Z)}\sum_{j=m}^{M} \frac{ a_{j+1}^{2\a}e^{-a_{j+1}^2/(4 s)}+a_j^{2\a} e^{-a_j^2/(4 s)}}{s^{1+\a}}  \abs{f_1(-s)}ds\\
  &\le C_{\alpha, v} \int_0^{+\infty} \sum_{j=m}^{M} \left(\frac{1}{a_{j+1}^{2}} + \frac{1}{a_{j}^{2}}  \right) \abs{f_1(-s)}ds\\
  &\le  C_{\alpha, v} (\rho^4+1)\int_0^{+\infty} \sum_{j=m}^{M}  \frac{1}{a_{j+1}^{2}} \abs{f_1(-s)}ds \quad (\text{since} \quad \rho\le \frac{a_{j+1}}{a_j}\le \rho^2)\\
  &\le C_{\alpha, v, \rho} \frac{1}{a_{m+1}^2} \int_0^{+\infty} \sum_{j=m}^{M}  \frac{a_{m+1}^2}{a_j^2}    \abs{f_1(-s)}ds \\
&\le C_{\alpha, v, \rho} \frac{1}{a_{m+1}^2} \int_0^{+\infty}  \Big( \rho^4+\sum_{j=m}^{M}  \frac{1}{\rho^{2(j-m)}}  \Big)  \abs{f_1(-s)}ds\\
&\le C_{\alpha, v, \rho} \frac{1}{a_{m+1}^2} \int_0^{+\infty} (\rho^2-1)\rho^{4\alpha} \Big( \rho^4+\sum_{j=0}^{+\infty}  \frac{1}{\rho^{2j}}  \Big)  \abs{f_1(-s)}ds\\
&\le C_{\alpha, v, \rho}\frac{1}{a_{m+1}^2}\int^0_{-a_{m+1}^2}   \abs{f(s)}~ds\\
&\le C_{\alpha,\rho, v}\M_q^-f(0).
\end{align*}
For part $II$,
\begin{align*}\label{equ:II}
II&=\abs{T_{m,M}^\alpha f_2(0)}=\frac{1}{a_m^2}\int_{-a_m^2}^0 \abs{T_{m,M}^\alpha f_2(0)}du \\
&\le \frac{1}{a_m^2} \int_{-a_m^2}^0\abs{T_{-M,M}^\alpha f(u)}du +\frac{1}{a_m^2} \int_{-a_m^2}^0\abs{T_{-M,M}^\alpha f_1(u)}du\\
&\quad +\frac{1}{a_m^2} \int_{-a_m^2}^0\abs{T_{m,M}^\alpha f_2(u)-T_{m,M}^\alpha f_2(0)}du\\
&\quad +\frac{1}{a_m^2} \int_{-a_m^2}^0\abs{T_{-M,m-1}^\alpha f_2(u)}du\\
&=:A_1+A_2+A_3+A_4.
\end{align*}
(If $m=-M$, we understand that $A_4=0$.) It is clear that
\begin{equation*}
A_1\le  \M^- (T_{-M,M}^\alpha f)(0).
\end{equation*}
For $A_2,$ by the uniform boundedness of $T^\alpha_{N}$ in Theorem \ref{Thm:PoissonLpOne}, we get
\begin{multline*}
A_2\le \left(\frac{1}{a_m^2}\int_{-a_m^2}^0\abs{T_{-M,M}^\alpha f_1(u)}^qdu\right)^{1/q}\le C\left(\frac{1}{a_m^2}\int_{\mathbb{R}}\abs{f_1(u)}^qdu\right)^{1/q}\\
=C\left(\frac{1}{a_m^2}\int_{-a_m^2}^0\abs{f(u)}^qdu\right)^{1/q}\le C\M^-_qf(0).
\end{multline*}
For the third term $A_3$, with $-a_m^2\le u\le 0$, by the mean value theorem and  Theorem \ref{Thm:KernelEstOne}, we have
\begin{align*}
&\abs{T_{m,M}^\alpha f_2(u)-T_{m,M}^\alpha f_2(0)}=\abs{\int_{-\infty}^u K_{m,M}^\alpha(u-s)f_2(s)ds-\int_{-\infty}^0 K_{m,M}^\alpha(-s)f_2(s)ds }\\
&\le \int_{-\infty}^u \abs{K_{m,M}^\alpha(u-s)-K_{m,M}^\alpha(-s)}\abs{f_2(s)}ds+\abs{\int_{u}^0K_{m,M}^\alpha(-s)f_2(s)ds}\\
&=\int_{-\infty}^{-a_{m+1}^2} \abs{K_{m,M}^\alpha(u-s)-K_{m,M}^\alpha(-s)}\abs{f(s)}ds\\
&=\sum_{j=m+1}^{+\infty}\int_{-a_{j+1}^2}^{-a_{j}^2} \abs{K_{m,M}^\alpha(u-s)-K_{m,M}^\alpha(-s)}\abs{f(s)}ds\\
&=\sum_{j=m+1}^{+\infty}\int_{-a_{j+1}^2}^{-a_{j}^2} \abs{\partial_t K_{m,M}^\alpha(t)\big|_{t=\xi_j}}\abs{u}\abs{f(s)}ds\quad ( a_j^2-a_m^2\le \xi_j\le a_{j+1}^2)\\
&\le C\sum_{j=m+1}^{+\infty}\int_{-a_{j+1}^2}^{-a_{j}^2} \frac{\abs{u}}{\abs{\xi_j}^2}\abs{f(s)}ds\le C\sum_{j=m+1}^{+\infty}\frac{a_m^2}{\left(a_j^2-a_m^2\right)^2}\int_{-a_{j+1}^2}^{0} \abs{f(s)}ds\\
&\le C\sum_{j=m+1}^{+\infty}\frac{a_m^2}{a_j^2}\cdot\frac{\rho^4}{(\rho^4-1)a_{j+1}^2}\int_{-a_{j+1}^2}^{0} \abs{f(s)}ds\\
&\le C\sum_{j=m+1}^{+\infty}\frac{1}{\rho^{2(j-m)}}\ \M^-f(0)\\
&\le C\M_q^-f(0).
\end{align*}
Then,
\begin{align*}
A_3=\frac{1}{a_m^2} \int_{-a_m^2}^0\abs{T_{m,M}^\alpha f_2(u)-T_{m,M}^\alpha f_2(0)}du \le C\M_q^-f(0).
\end{align*}
For the latest one,  $A_4,$  we have
\begin{align*}
A_4=\frac{1}{a_m^2} \int_{-a_m^2}^0\abs{T_{-M,m-1}^\alpha f_2(u)}du\le \frac{1}{a_m^2} \int_{-a_m^2}^0\int_{-\infty}^{-a_{m+1}^2}
\abs{K_{-M,m-1}^\alpha (u-s)f_2(s)}dsdu.
\end{align*}
Then, we consider the inner integral appeared in the above inequalities first. Since $-a_m^2\le u\le 0$,  $-\infty<s\le -a_{m+1}^2$ and the sequence $\{a_j\}_{j\in \mathbb Z}$ is $\rho$-lacunary sequence, we have  $\abs{u-s}\sim \abs{s}.$
From this and by the mean value theorem, we get
\begin{align*}
&\quad \int_{-\infty}^{-a_{m+1}^2} \abs{K_{-M,m-1}^\alpha (u-s)f_2(s)}ds\\
&=\sum_{k=m+1}^{+\infty}\int_{-a_{k+1}^2}^{-a_{k}^2} \abs{\sum_{j=-M}^{m-1}v_j \frac{ a_{j+1}^{2\a}e^{-a_{j+1}^2/(4 (u-s))}-a_j^{2\a} e^{-a_j^2/(4 (u-s))}}{(u-s)^{1+\a}} f(s)}ds\\
&\le \sum_{k=m+1}^{+\infty}\int_{-a_{k+1}^2}^{-a_{k}^2} \abs{\sum_{j=-M}^{m-1}v_j \frac{ (a_{j+1}-a_j)\xi_{j}^{2\alpha-1}e^{-\xi_{j}^2/(4 (u-s))}}{(u-s)^{1+\a}} f(s)}ds \quad (a_j\le \xi_j\le a_{j+1})\\
&\le C\norm{v}_{l^\infty(\mathbb Z)} \sum_{k=m+1}^{+\infty}\int_{-a_{k+1}^2}^{-a_{k}^2}\sum_{j=-M}^{m-1} \abs{\frac{\rho^{4\alpha} (\rho^2-1)a_{j}^{2\alpha}e^{-a_{j}^2/(4 s)}}{s^{1+\a}}}\abs{ f(s)}ds\\
&\le C_{\rho, v, \alpha} \sum_{k=m+1}^{+\infty}{1\over {a_{k}^2}}\int_{-a_{k+1}^2}^{-a_{k}^2}\sum_{j=-M}^{m-1} \frac{a_{j}^{2\alpha}}{a_k^{2\a}}\abs{ f(s)}ds\\
&\le C_{\rho, v, \alpha} \sum_{k=m+1}^{+\infty}{1\over {a_{k+1}^2}}\int_{-a_{k+1}^2}^{-a_{k}^2}\sum_{j=-M}^{m-1} {\rho^{-2\alpha(k-j)}}\abs{ f(s)}ds\\
&\le C_{\rho, v, \alpha} \sum_{k=m+1}^{+\infty}\frac{\rho^{-2\alpha(k-m+1)}}{a_{k+1}^2}\int_{-a_{k+1}^2}^{-a_{k}^2}\abs{ f(s)}ds\\
&\le C_{\rho, v, \alpha} \sum_{k=m+1}^{+\infty}\frac{1}{\rho^{2\alpha(k-m+1)}}\frac{1}{a_{k+1}^2} \int_{-a_{k+1}^2}^{0}\abs{ f(s)}ds\\
&\le C_{\rho, v, \alpha} \sum_{k=m+1}^{+\infty}\frac{1}{\rho^{2\alpha(k-m+1)}}\M^-f(0)\\
&\le C_{\rho, v, \alpha}\M_q^-f(0).\\
\end{align*}
Hence,
$$A_4\le C\M_q^-f(0).$$
Combining the estimates above for $A_1, A_2, A_3$ and $A_4$, we get
$$II\le \M^- (T_{-M,M}^\alpha f)(0)+C \M_q^- f(0).$$
And, then we have
$$ \abs{T_{m,M}^\alpha f(0)}\le C\left( \M^- (T_{-M,M}^\alpha f)(0)+\M_q^- f(0)\right). $$
As the constants $C$ appeared above all  depend on $\norm{v}_{l^\infty(\mathbb Z)}$, $\rho$ and $\alpha$, not on $m, M$, we complete the proof.
\end{proof}

Now we can start the proof of Theorem \ref{thm:Maxi}.

\begin{proof}[Proof of Theorem \ref{thm:Maxi}]
For each $\omega\in A_p^-,$ choose $1<q<p<\infty$ such that $\omega\in A_{p/q}^-.$  Then, it is well known that the   maximal operators $\M^-$ and $ \M_q^-$ are bounded in $L^p(\real, \omega)$.  On the other hand, by Theorem \ref{Thm:PoissonLpOne}, the operators $T_N^\alpha$ are uniformly bounded in $L^p(\real, \omega)$ with $\omega \in A_p^-$. Hence
\begin{align*}
\norm{T_M^*f}_{L^p(\omega)}&\le C\left(\norm{\M^- (T_{-M, M}^\alpha f)}_{L^p(\omega)}+\norm{\M_q^- f}_{L^p(\omega)}\right)\\&\le C\left(\norm{T_{-M, M}^\alpha f}_{L^p(\omega)}+\norm{f}_{L^p(\omega)}\right)\le C\norm{f}_{L^p(\omega)}.
\end{align*}
Note that the constants $C$ appeared above do not depend on $M$. Consequently, letting $M$ increase to infinity, we get the proof of the $L^p$ boundedness of $T^*.$  This completes the proof of part $(a)$ of the theorem.

In order to prove $(b)$, we consider the $\ell^\infty(\mathbb Z^2)$-valued operator
$\mathcal{T}f(t) = \{  T_N^\alpha f(t) \}_{N\in \mathbb Z^2}$. Since $\|\mathcal{T}f(t) \|_{\ell^\infty(\mathbb Z^2))}= T^*f(t)$,   by using $(a)$ we know that the operator $\mathcal{T}$ is bounded from $L^p(\mathbb{R}, \omega) $ into $L^p_{\ell^\infty(\mathbb Z^2)}(\mathbb{R}, \omega) $, for every $1<p<\infty$ and $\omega \in A_p^-$. The kernel of the operator $\mathcal{T}$ is given by $\mathcal{K}^\alpha(t) = \{ K^\alpha_N(t)\} _{N\in \mathbb Z^2}$. By Theorem \ref{Thm:KernelEstOne} and the vector valued version of Theorem 2.1 in \cite{AFM}, we get that the operator $\mathcal{T}$ is bounded from $L^1(\mathbb{R}, \omega)$ into weak- $L^1_{\ell^\infty(\mathbb Z^2)}(\mathbb{R}, \omega)$ for $\omega \in A_1^-$. Hence, as $\|\mathcal{T}f(t) \|_{\ell^\infty(\mathbb Z^2)}= T^*f(t)$, we get the proof of  $(b)$.

For $(c),$ we shall prove that if $f\in L^\infty(\mathbb R)$ and there exists $t_0\in \mathbb R$ such that $T^*f(t_0)<\infty,$ then $T^*f(t)<\infty$ for $a.e.$ $t\in \real.$ Given $t\neq t_0.$ Set
$f_1=f\chi_{( t_0-4|t_0-t|,\  t_0+4|t_0-t|)}$ and $f_2 = f-f_1$.   Note that $T^*$ is $L^p$-bounded for any $1<p<\infty.$ Then $T^*f_1(t)<\infty$, because $f_1\in L^p(\mathbb R)$, for any $1<p<\infty.$ On the other hand, as the kernel $K_N$ is supported in $\mathbb{R}^+,$ we have
\begin{align*}
&\Big|T_N^\alpha f_2(t)-T_N^\alpha f_2(t_0)\Big|\\&=\Big|\int_{-\infty}^{t} K_N^\alpha (t-s)f_2(s)ds-\int_{-\infty}^{t_0} K_N^\alpha (t_0-s)f_2(s)ds \Big|\\
&=\Big| \int_{-\infty}^{t_0-4|t_0-t|}\left(K_N^\alpha(t-s) - K_N^\alpha (t_0-s)\right)f_2(s)ds\Big|\\
&\le \int_{-\infty}^{t_0-4|t_0-t|} \abs{\partial_s K_N^\alpha(\xi(s))}\abs{t-t_0}\abs{f_2(s)}ds \quad  (t-s\le \xi(s)\le t_0-s) \\
&\le C \int_{-\infty}^{{t_0-4|t_0-t|}  }  \frac{\abs{t-t_0}}{(t-s)^2}\abs{f_2(s)}ds  \\
&\le C\norm{f}_{L^\infty(\mathbb R)}< \infty.
 \end{align*}
 Hence
  \begin{align*}
\norm{T_N^\alpha f_2(t)-T_N^\alpha f_2(t_0)}_{l^\infty(\mathbb Z^2)} \le C\norm{f}_{L^\infty(\mathbb R)}
\end{align*}
and therefore $ T^*f(t) = \norm{T_N^\alpha f(t)}_{l^\infty(\mathbb Z^2)}  \le C < \infty.$ For the $L^\infty-BMO$ boundedness, we will prove it later.

$(d)$ Let $t_0$ be one point in $\real$ such that  $T^*f(t_0)<\infty.$
Set $I=[t_0-4|t_0-t|,\  t_0+4|t_0-t|]$ with $t\neq t_0$. And we decompose $f$ to be
\begin{equation*}
f=(f-f_I)\chi_I+(f-f_I)\chi_{I^c}+f_I=:f_1+f_2+f_3.
\end{equation*}
Note that $T^*$ is $L^p$-bounded for any $1<p<\infty.$ Then $T^*f_1(t)<\infty$, because $f_1\in L^p(\mathbb R)$, for any $1<p<\infty.$ And  $T_N^\alpha f_3=0$, since $\P_{a_j}^\alpha f_3=f_3$ for any $j\in \mathbb Z.$
On the other hand, as the kernel $K_N$ is supported in $\mathbb{R}^+,$ we have
\begin{align*}
&\Big|T_N^\alpha f_2(t)-T_N^\alpha f_2(t_0)\Big|\\&=\Big|\int_{-\infty}^{t} K_N^\alpha (t-s)f_2(s)ds-\int_{-\infty}^{t_0} K_N^\alpha (t_0-s)f_2(s)ds \Big|\\
&=\Big| \int_{-\infty}^{t_0-4|t_0-t|}\left(K_N^\alpha(t-s) - K_N^\alpha (t_0-s)\right)f_2(s)ds\Big|\\
&\le \int_{-\infty}^{t_0-4|t_0-t|} \abs{\partial_s K_N^\alpha(\xi(s))}\abs{t-t_0}\abs{f_2(s)}ds \quad  (t-s\le \xi(s)\le t_0-s) \\
&\le C \int_{-\infty}^{{t_0-4|t_0-t|}  }  \frac{\abs{t-t_0}}{(t-s)^2}\abs{f_2(s)}ds  \\
& \le C \sum_{k=2}^{+\infty}{ |t-t_0|} \int_{t_0-2^{k+1}|t_0-t|}^{{t_0-2^k|t_0-t|}}  {\abs{f(s)-f_I}\over |t-s|^2}ds\\
&\le C \sum_{k=2}^{+\infty}{ |t-t_0|\over (2^{k+1}|t-t_0|)^2} \int_{t_0-2^{k+1}|t_0-t|}^{{t_0-2^k|t_0-t|}}  {\abs{f(s)-f_I}}ds\\
&\le  C \sum_{k=2}^{+\infty}{ |t-t_0|\over (2^{k+1}|t-t_0|)^2} \int_{t_0-2^{k+1}|t_0-t|}^{{t_0+2^{k+1}|t_0-t|}}  {\abs{f(s)-f_I}}ds\\
&=C \sum_{k=2}^{+\infty}2^{-(k+1)}{ 1\over 2^{k+1}|t-t_0|} \int_{I_{k+1}}  {\abs{f(s)-f_I}}ds\\
&\le C \sum_{k=2}^{+\infty}2^{-(k+1)}{ 1\over 2^{k+1}|t-t_0|} \int_{I_{k+1}} \left(  {\abs{f(s)-f_{I_{k+1}}}}+\sum_{l=2}^{k}\abs{f_{I_{l+1}}-f_{I_{l}}}\right)ds\\
&\le C\sum_{k=2}^{+\infty}2^{-(k+1)}{ 1\over 2^{k+1}|t-t_0|} \int_{I_{k+1}}\left(  {\abs{f(s)-f_{I_{k+1}}}}+2k\norm{f}_{BMO(\real)}\right)ds\\
&\le C\sum_{k=2}^{+\infty}2^{-(k+1)}{(1+2k)\norm{f}_{BMO(\real)}}\\
&\le  C\norm{f}_{BMO(\real)},
\end{align*}
where $I_{k+1}=[t_0-2^{k+1}|t_0-t|, {t_0+2^{k+1}|t_0-t|}]$ for any $k\in \mathbb N.$
 Hence
  \begin{align*}
\norm{T_N^\alpha f_2(t)-T_N^\alpha f_2(t_0)}_{l^\infty(\mathbb Z^2)} \le C\norm{f}_{BMO(\mathbb R)}
\end{align*}
and therefore $ T^*f(t) = \norm{T_N^\alpha f(t)}_{l^\infty(\mathbb Z^2)}  \le C < \infty.$

Now, we shall prove  the  estimate (\ref{sharp}) for functions such that $T^*f(t) < \infty \, \, a.e.$
For any $h>0$ and $t_0$ such that $T^*f(t_0) < \infty$, consider the integral $I=(t_0, t_0+h)$ and $\displaystyle f_I={1\over h}\int_I f(t)dt.$  We have $T^*f_I(t)=0.$ Let $f(t)=f_1(t)+f_2(t)+f_I$, where $f_1(t)=(f(t)-f_I)\chi_{(t_0-4h, t_0+4h)}(t)$ and $f_2(t)=(f(t)-f_I)\chi_{(-\infty, t_0-4h)}(t)+(f(t)-f_I)\chi_{(t_0+4h, +\infty)}(t)$.
Then,
\begin{align*}
&{1\over h}\int_{t_0}^{t_0+h}\abs{T^*f(t)-(T^*f)_I}dt={1\over h}\int_{t_0}^{t_0+h}\abs{{1\over h}\int_{t_0}^{t_0+h}\left(T^*f(t)-T^*f(s)\right)ds}dt\\
&\le {1\over h^2}\int_{t_0}^{t_0+h}\int_{t_0}^{t_0+h}\abs{T^*f(t)-T^*f(s)}dsdt\\
&={1\over h^2}\int_{t_0}^{t_0+h}\int_{t_0}^{t_0+h}\abs{\norm{T^\alpha_Nf(t)}_{l^\infty(\mathbb Z^2)}-\norm{T^\alpha_Nf(s)}_{l^\infty(\mathbb Z^2)}}dsdt\\
&\le {1\over h^2}\int_{t_0}^{t_0+h}\int_{t_0}^{t_0+h}{\norm{T^\alpha_Nf(t)-T^\alpha_Nf(s)}_{l^\infty(\mathbb Z^2)}}dsdt\\
&\le {1\over h^2}\int_{t_0}^{t_0+h}\int_{t_0}^{t_0+h}{\norm{T^\alpha_Nf_1(t)-T^\alpha_N f_1(s)}_{l^\infty(\mathbb Z^2)}}dsdt\\
&\quad + {1\over h^2}\int_{t_0}^{t_0+h}\int_{t_0}^{t_0+h}{\norm{T^\alpha_Nf_2(t)-T^\alpha_Nf_2(s)}_{l^\infty(\mathbb Z^2)}}dsdt\\
&=:A+B.
\end{align*}
The H\"older inequality and  $L^2$-boundedness of $T^*$ imply that
\begin{align*}
A&\le {1\over h}\int_{t_0}^{t_0+h}{\norm{T^\alpha_Nf_1(t)}_{l^\infty(\mathbb Z^2)}}dt
      +{1\over h}\int_{t_0}^{t_0+h}{\norm{T^\alpha_Nf_1(s)}_{l^\infty(\mathbb Z^2)}}ds\\
  &\le \left({1\over h}\int_{t_0}^{t_0+h}{\norm{T^\alpha_Nf_1(t)}^2_{l^\infty(\mathbb Z^2)}}dt\right)^{1/2}
      +\left({1\over h}\int_{t_0}^{t_0+h}{\norm{T^\alpha_Nf_1(s)}^2_{l^\infty(\mathbb Z^2)}}ds\right)^{1/2}\\
  &\le C{1\over h^{1/2}}\norm{f_1}_{L^2(\mathbb R)}\le  C\norm{f}_{BMO(\mathbb R)}.
\end{align*}
For $B$, since $t_0\le t,s\le t_0+h$ and the support of $f_2$ is $(-\infty, t_0-4h)\bigcup(t_0+4h, +\infty)$,
we have
\begin{align*}
&\Big|T_N^\alpha f_2(t)-T_N^\alpha f_2(s)\Big|\\
&=\Big|\int_{-\infty}^{t} K_N^\alpha (t-u)f_2(u)du-\int_{-\infty}^{t_0} K_N^\alpha (s-u)f_2(u)du \Big|\\
&=\Big| \int_{-\infty}^{t_0-4h}\left(K_N^\alpha(t-u) - K_N^\alpha (s-u)\right)f_2(u)du\Big|\\
&\le \int_{-\infty}^{t_0-4h} \abs{\partial_u K_N^\alpha(\xi(u))}\abs{t-s}\abs{f_2(u)}du \quad  (t-u\le \xi(u)\le t_0-u) \\
&\le C \int_{-\infty}^{{t_0-4h}  }  \frac{\abs{t-s}}{(t-u)^2}\abs{f_2(u)}du  \\
& \le C \sum_{k=2}^{+\infty
} \int_{t_0-2^{k+1}h}^{{t_0-2^kh}}  {{ h}\abs{f(u)-f_I}\over |t-u|^2}du\\
&\le C \sum_{k=2}^{+\infty}{ h\over (2^{k+1}h)^2} \int_{t_0-2^{k+1}h}^{{t_0-2^kh}}  {\abs{f(u)-f_I}}du\\
&\le  C \sum_{k=2}^{+\infty}{ h\over (2^{k+1}h)^2} \int_{t_0-2^{k+1}h}^{{t_0+2^{k+1}h}}  {\abs{f(u)-f_I}}du\\
&=C \sum_{k=2}^{+\infty}2^{-(k+1)}{ 1\over 2^{k+1}h} \int_{I_{k+1}}  {\abs{f(u)-f_I}}du\\
&\le C \sum_{k=2}^{+\infty}2^{-(k+1)}{ 1\over 2^{k+1}h} \int_{I_{k+1}} \left(  {\abs{f(u)-f_{I_{k+1}}}}+\sum_{l=2}^{k}\abs{f_{I_{l+1}}-f_{I_{l}}}\right)du\\
&\le C\sum_{k=2}^{+\infty}2^{-(k+1)}{ 1\over 2^{k+1}h} \int_{I_{k+1}}\left(  {\abs{f(u)-f_{I_{k+1}}}}+2k\norm{f}_{BMO(\real)}\right)du\\
&\le C\sum_{k=2}^{+\infty}2^{-(k+1)}{(1+2k)\norm{f}_{BMO(\real)}}\\
&\le  C\norm{f}_{BMO(\real)},
\end{align*}
where $I_{k+1}$ denotes the interval $[t_0-2^{k+1}h, t_0+2^{k+1}h]$.  Hence, we have $B\le C \norm{f}_{BMO(\real)}.$
Then by the arbitrary of $t_0$ and $h>0$, we proved
$$\norm{T^*f}_{BMO(\mathbb R)}\le C\norm{f}_{BMO(\mathbb R)}.$$
For the second part of $(c)$, we can deduce it from the $BMO$-boundedness of $T^*$ and the inclusion of $L^\infty(\real)\subset BMO(\real).$
This completes the proof of Theorem \ref{thm:Maxi}.
\end{proof}

Now we shall prove Theorem \ref{Thm:ae}.
\begin{proof}[Proof of Theorem \ref{Thm:ae}]
First, we shall see that if $\varphi$ is a test function, then $T_N^\alpha \varphi(t)$ converges for all $t\in \real$.  In order to prove this, it is enough to see that for any  $(L,M)$ with $0<L<M$,  the  series \begin{equation*} A= \sum_{j=L}^M v_j ( \P^\alpha_{a_{j+1}} \varphi(t) - \P_{a_j}^\alpha \varphi(t))  \hbox{  and  } B= \sum_{j=-M}^{-L} v_j ( \P^\alpha_{a_{j+1}} \varphi(t) - \P_{a_j}^\alpha \varphi(t))
 \end{equation*}
converge to zero, when $L, M\rightarrow +\infty$.
By the mean value theorem, following the arguments in the proof of Theorem \ref{Thm:Maximalcontrol}, we have
\begin{align*}
|A|   & \le  C_\alpha \norm{v}_{l^\infty(\mathbb Z)} \int_0^\infty\sum_{j=L}^{M} \abs{\frac{\xi_j^{2\alpha-1} e^{-\xi_{j}^2/(4 s)}(a_{j+1}-a_j)}{s^{1+\a}}} |\varphi(t-s)|ds, \quad (\exists \ a_j\le \xi_j\le a_{j+1}) \\
&\le C_{\alpha, v}  \int_0^{+\infty} \rho^{4\alpha}(\rho^2-1)\sum_{j=L}^{M}  \frac{a_{j}^{2\alpha} e^{-a_{j}^2/(4 s)}}{s^{1+\a}}  |\varphi(t-s)|ds, \quad (\text{since} \quad \rho\le \frac{a_{j+1}}{a_j}\le \rho^2)\\
&\le C_{\alpha, v, \rho}   \int_0^{+\infty} \sum_{j=L}^{M}  \frac{C}{a^2_j} |\varphi(t-s)|ds \\
&\le  C_{\alpha, v, \rho}\left({1\over a_L^2}\sum_{j=L}^{M} \frac{a_L^2}{a_j^2}\right) \int_0^{+\infty} |\varphi(t-s)|ds\\
&\le C_{\alpha,v, \rho} {\rho^2\over {\rho^2-1}}\|\varphi\|_{L^1(\real)}{1\over a_L^2} \longrightarrow 0, \quad \hbox{as}\ { L,M \to +\infty}.
\end{align*}
On the other hand,  as the integral of the kernels are zero, we can write
\begin{align*}
B&=C_\alpha \int_0^{+\infty} \sum_{j=-M}^{-L}v_j \frac{ a_{j+1}^{2\a}e^{-a_{j+1}^2/(4 s)}-a_j^{2\a} e^{-a_j^2/(4 s)}}{s^{1+\a}} ( \varphi(t-s)- \varphi(t))ds\\
&= C_\alpha \left\{\int_0^1 + \int_1^\infty\right\} \sum_{j=-M}^{-L}v_j \frac{ a_{j+1}^{2\a}e^{-a_{j+1}^2/(4 s)}-a_j^{2\a} e^{-a_j^2/(4 s)}}{s^{1+\a}} ( \varphi(t-s)- \varphi(t))ds\\
&=: B_1+B_2.  \end{align*}
Proceeding as in the case $A$, and by using the fact that $\varphi$ is a test function,  we have
\begin{align*}
|B_1|&= C_\alpha \Big| \int_0^{1} \sum_{j=-M}^{-L}v_j \frac{ a_{j+1}^{2\a}e^{-a_{j+1}^2/(4 s)}-a_j^{2\a} e^{-a_j^2/(4 s)}}{s^{1+\a}} ( \varphi(t-s)- \varphi(t)) ds\Big| \\
&\le C_\alpha \norm{\varphi'}_{L^\infty(\real)} \int_0^{1} \sum_{j=-M}^{-L}v_j \frac{ a_{j+1}^{2\a}e^{-a_{j+1}^2/(4 s)}-a_j^{2\a} e^{-a_j^2/(4 s)}}{s^{\a}} ds \\
&\le C_{\alpha, \varphi} \norm{v}_{l^\infty (\mathbb Z)}  \int_0^1 \rho^{4\alpha}(\rho^2-1)\sum_{j=-M}^{-L}  \frac{a_{j}^{2\alpha} e^{-a_{j}^2/(4 s)}}{s^{\a}}  ds\\
&\le C_{\alpha, \varphi, v, \rho}  a_{-L}^{2\alpha}\sum_{j=-M}^{-L}  {a_j^{2\alpha}\over  a_{-L}^{2\alpha}} \int_0^1\frac1{s^{\a}}  ds \\
 &\le C_{\alpha, \varphi, v, \rho}  {\rho^{2\alpha}\over \rho^{2\alpha}-1}a_{-L}^{2\alpha}   \longrightarrow 0,\quad  \hbox{as}\quad { L,M \to +\infty}. \end{align*}
On the other hand,
\begin{align*}
|B_2|&\le C_{\alpha,\rho} \norm{v}_{l^\infty (\mathbb Z)} \|\varphi\|_{L^\infty(\real) }\int_1^\infty \sum_{j=-M}^{-L}\frac{ a_{j}^{2\a}}{s^{1+\a}}  ds \le  C_{\alpha, v, \varphi,\rho} \sum_{j=-M}^{-L}{ a_j^{2\alpha}}\int_1^\infty {1\over s^{1+\a}}  ds \\
& \le C_{\alpha, v, \varphi,\rho}    a_{-L}^{2\alpha}\sum_{j=-M}^{-L}  {a_j^{2\alpha}\over  a_{-L}^{2\alpha}} \le C_{\alpha, \varphi, v,\rho}  {\rho^{2\alpha}\over \rho^{2\alpha}-1}a_{-L}^{2\alpha}   \longrightarrow 0,\quad \hbox{as}\quad { L,M \to +\infty}.
\end{align*}

As the set of test functions is dense in $L^p(\mathbb{R})$, by Theorem \ref{thm:Maxi} we get the $a.e.$ convergence for any function in $L^p(\mathbb{R})$. Analogously, since $L^p(\mathbb{R}) \cap L^p(\mathbb{R}, \omega) $ is dense in $L^p(\mathbb{R}, \omega)$, we get the $a.e.$ convergence for functions in $L^p(\mathbb{R}, \omega) $ with $1\le p<\infty$. By using the dominated convergence theorem,  we can prove the convergence in $L^p(\mathbb{R}, \omega)$-norm for $1<p<\infty$, and also in measure.
\end{proof}

\medskip

\section{Proofs of Theorems \ref{Thm:infny} and  \ref{Thm:GrothLinfinity}}\label{sec:equinfty}
In this section, we will give the proof of  Theorems \ref{Thm:infny} and  \ref{Thm:GrothLinfinity}.

\begin{proof}[Proof of Theorem \ref{Thm:infny}]
 Let  $f$ be the function defined by
\begin{equation*}
f(s) = \sum_{k\in \mathbb{Z}} (-1)^{k} \chi_{(-a^{2k+1}, -a^{2k}]}(s),
\end{equation*}
where $a>1$ is a real number that we shall fix it later.
It is easy to see that
\begin{equation}\label{equ:dilation}
f(a^{2j}s) = (-1)^j f(s).
\end{equation}
Let $a_j= a^{j}.$ Then
\begin{align*}
\mathcal{P}_{a_j}^\alpha f(t) &= \frac{1}{4^\alpha \Gamma(\alpha) }
\int_0^{+\infty} \frac{a^{2\alpha j} e^{-a^{2j}/(4s)}}{s^{1+\alpha}} f(t-s) ds =
\frac{1}{4^\alpha \Gamma(\alpha) }
\int_0^{+\infty} \frac{ e^{-1/(4u)}}{u^{\alpha}} f(t-a^{2j}u) \frac{du}{u}.
\end{align*}
So
\begin{equation*}
\mathcal{P}_{a_j}^\alpha f(0) =
\frac{1}{4^\alpha \Gamma(\alpha) }
\int_0^{+\infty} \frac{ e^{-1/(4u)}}{u^{\alpha}} f(-a^{2j}u) \frac{du}{u} =
(-1)^j \frac{1}{4^\alpha \Gamma(\alpha) }
\int_0^{+\infty} \frac{ e^{-1/(4u)}}{u^{\alpha}} f(-u) \frac{du}{u}.
\end{equation*}
We observe that
\begin{eqnarray*}
\int_0^{+\infty} \frac{ e^{-1/(4u)}}{u^{\alpha}} \big| f(-u)\big| \frac{du}{u} \le
\int_0^{+\infty} \frac{ e^{-1/(4u)}}{u^{\alpha}} \frac{du}{u}=4^\alpha\Gamma(\alpha)< \infty.
\end{eqnarray*}
Hence $$\lim_{R\to {+\infty}} \int_R^{+\infty} \frac{ e^{-1/(4u)}}{u^{\alpha}}  f(-u)\frac{du}{u} =0\quad  \hbox{  and  } \quad \lim_{\varepsilon \to 0^+} \int_0^\varepsilon \frac{ e^{-1/(4u)}}{u^{\alpha}}  f(-u)\frac{du}{u} =0. $$
On the other hand, $\displaystyle \lim_{a\to {+\infty}} \int_1^a \frac{ e^{-1/(4u)}}{u^{\alpha}}  f(-u)\frac{du}{u}  = \displaystyle \lim_{a\to {+\infty}} \int_1^a \frac{ e^{-1/(4u)}}{u^{\alpha}} \frac{du}{u} =C>0.$  Hence we can choose $a>1$ big enough such that
 \begin{multline*}
 \int_1^a \frac{ e^{-1/(4u)}}{u^{\alpha}}  f(-u)\frac{du}{u} = \int_1^a \frac{ e^{-1/(4u)}}{u^{\alpha}}  \frac{du}{u}  >  \abs{\int_0^{1/a} \frac{ e^{-1/(4u)}}{u^{\alpha}} \frac{du}{u}} + \abs{\int_{a^2}^{+\infty} \frac{ e^{-1/(4u)}}{u^{\alpha}} \frac{du}{u}}\\
 >  \abs{\int_0^{1/a} \frac{ e^{-1/(4u)}}{u^{\alpha}}f(-u) \frac{du}{u}} + \abs{\int_{a^2}^{+\infty} \frac{ e^{-1/(4u)}}{u^{\alpha}}f(-u) \frac{du}{u}}.
\end{multline*}
In other words, with the $a>1$ fixed above,  there exists constant $C_1>0$ such that
\begin{equation}\label{equ:cons1}
\int_0^{+\infty} \frac{ e^{-1/(4u)}}{u^{\alpha}}  f(-u)\frac{du}{u}=C_1.
\end{equation}
Hence
\begin{equation*}
\Big|\mathcal{P}_{a_j}^\alpha f(0)- \mathcal{P}_{a_{j+1}}^\alpha f(0)\Big| = \frac{ 2 C_1}{4^\alpha \Gamma(\alpha)}>0.
\end{equation*}
Therefore we have   $$\sum_{j\in \mathbb Z} \Big|\mathcal{P}_{a_{j+1}}^\alpha f(0)- \mathcal{P}_{a_{j}}^\alpha f(0)\Big|  = \infty.$$
By using \eqref{equ:dilation} and changing variable  we get
\begin{equation*}
\mathcal{P}_{a_j}^\alpha f(t) =
\frac{1}{4^\alpha \Gamma(\alpha) }
\int_0^{+\infty} \frac{ e^{-1/(4u)}}{u^{\alpha}} f(t-a^{2j}u) \frac{du}{u} =
(-1)^j \frac{1}{4^\alpha \Gamma(\alpha) }
\int_0^{+\infty} \frac{ e^{-1/(4u)}}{u^{\alpha}} f\left({t\over a^{2j}}-u\right) \frac{du}{u}.
\end{equation*}
Then
\begin{align}\label{equ:diffP}
&{\mathcal{P}_{a_{j+1}}^\alpha f(t)-\mathcal{P}_{a_j}^\alpha f(t)}\nonumber\\
& =
\frac{ (-1)^{j+1}}{4^\alpha \Gamma(\alpha) }
\Big\{\int_0^{+\infty}\frac{ e^{-1/(4u)}}{u^{\alpha}} f\left({t\over {a^{2(j+1)}}}-u\right) \frac{du}{u}+\int_0^{+\infty}\frac{ e^{-1/(4u)}}{u^{\alpha}} f\left({t\over {a^{2j}}}-u\right) \frac{du}{u}\Big\}.
\end{align}
By the dominated convergence theorem, we know that
\begin{equation*}
\lim_{h\to 0}\int_0^{+\infty}\frac{ e^{-1/(4u)}}{u^{\alpha}} f(h-u) \frac{du}{u}=\int_0^{+\infty}\frac{ e^{-1/(4u)}}{u^{\alpha}} f(-u) \frac{du}{u}  = C_1>0,
\end{equation*}
where $C_1$ is the constant appeared in \eqref{equ:cons1}.
So, there exists $0<\eta_0<1,$ such that, for $|h|<\eta_0,$
\begin{equation*}
\int_0^{+\infty}\frac{ e^{-1/(4u)}}{u^{\alpha}} f(h-u) \frac{du}{u}\ge {1\over 2}\int_0^{+\infty}\frac{ e^{-1/(4u)}}{u^{\alpha}} f(-u) \frac{du}{u} = \frac{C_1}{2}.
\end{equation*}
Then, for each $t\in \real$, we can choose $j\in \mathbb Z$  such that $\displaystyle {|t|\over {a^j}}<\eta_0$  (there are infinite $j$ satisfying this condition), and we have
\begin{equation*}
\int_0^{+\infty}\frac{ e^{-1/(4u)}}{u^{\alpha}} f\left({t\over {a^{2(j+1)}}}-u\right) \frac{du}{u}+\int_0^{+\infty}\frac{ e^{-1/(4u)}}{u^{\alpha}} f\left({t\over {a^{2j}}}-u\right) \frac{du}{u} \ge C_1>0.
\end{equation*}

Choosing $v_j=(-1)^{j+1},\ j\in \mathbb Z$, by \eqref{equ:diffP} we have, for any $t\in \mathbb R,$
\begin{align*}
T^* f(t)&\ge \sum_{\abs{t\over {a^j}}<\eta_0} (-1)^{j+1}\big(\P_{a_{j+1}}^\alpha f(t)-\P_{a_j}^\alpha f(t)\big)\\
&=\frac{ 1}{4^\alpha \Gamma(\alpha) }\sum_{\abs{t\over {a^j}}<\eta_0}
\left({\int_0^{+\infty}\frac{ e^{-1/(4u)}}{u^{\alpha}} f\left({t\over {a^{2(j+1)}}}-u\right) \frac{du}{u}+\int_0^{+\infty}\frac{ e^{-1/(4u)}}{u^{\alpha}} f\left({t\over {a^{2j}}}-u\right) \frac{du}{u}}\right)\\
&=\infty.
\end{align*}
We complete the proof of Theorem \ref{Thm:infny}.
\end{proof}

\medskip

Also, we will give the proof of Theorem \ref{Thm:GrothLinfinity} which gives a local growth  characterization of the operator $T^*$ with $f\in L^\infty(\mathbb R^n)$.
\begin{proof}[Proof of Theorem \ref{Thm:GrothLinfinity}.]
 First, we prove the theorem in the case $1<p<\infty.$ Since $2r<1,$ we know that $B\backslash B_{2r}\neq \emptyset.$ Let $f(t)=f_1(t)+f_2(t)$, where $f_1(t)=f(t)\chi_{B_{2r}}(t)$ and $f_2(t)=f(t)\chi_{B\backslash B_{2r}}(t)$. Then
$$\abs{T^* f(t)}\le \abs{T^* f_1(t)}+\abs{T^* f_2(t)}.$$
By Theorem \ref{thm:Maxi},
\begin{multline*}
\frac{1}{|B_r|} \int_{B_r} \abs{T^* f_1 (t)} dt\le \left(\frac{1}{|B_r|} \int_{B_r} \abs{T^* f_1 (t)}^2 dt\right)^{1/2}
\le C \left(\frac{1}{|B_r|} \int_{\mathbb{R}}\abs{ f_1 (t)}^2 dt\right)^{1/2}\le C\norm{f}_{L^\infty(\mathbb{R})}.
\end{multline*}
We also know that, for any $j\in \mathbb Z,$
 \begin{multline}\label{equ:constant}
 \int_0^\infty\abs{\frac{ a_{j+1}^{2\a}e^{-a_{j+1}^2/(4 s)}-a_j^{2\a} e^{-a_j^2/(4 s)}}{s^{1+\a}}  }~ ds\\
 \le \int_0^\infty{\frac{ a_{j+1}^{2\a}e^{-a_{j+1}^2/(4 s)}+a_j^{2\a} e^{-a_j^2/(4 s)}}{s^{1+\a}}  }~ ds
 =2\cdot 4^\alpha \Gamma(\alpha).
 \end{multline}
 Then, by H\"older's inequality, \eqref{equ:constant}
 and Fubini's Theorem, for $1< p < \infty$ and any $N=(N_1, N_2)$, we have
\begin{align*}\label{equ:Palpha}
&\abs{\sum_{j=N_1}^{N_2}v_j\left(\P_{a_{j+1}}^\a f_2(t)-\P_{a_j}^\a f_2(t)\right)}\nonumber\\
&\le C\sum_{j=N_1}^{N_2} \abs{v_j \int_0^\infty\frac{ a_{j+1}^{2\a}e^{-a_{j+1}^2/(4 s)}-a_j^{2\a} e^{-a_j^2/(4 s))}}{s^{1+\a}} f_2(t-s)~ ds}\nonumber\\
&\le C\norm{v}_{l^p(\mathbb Z)}\left( \sum_{j=N_1}^{N_2}\left(\int_0^\infty\abs{\frac{ a_{j+1}^{2\a}e^{-a_{j+1}^2/(4 s)}-a_j^{2\a} e^{-a_j^2/(4 s)}}{s^{1+\a}}  } \abs{f_2(t-s)}~ ds\right)^{p'}\right)^{1/p'} \nonumber \\
&\le C\norm{v}_{l^p(\mathbb Z)}\Big(\sum_{j=N_1}^{N_2} \Big\{\int_0^\infty\abs{\frac{ a_{j+1}^{2\a}e^{-a_{j+1}^2/(4 s)}-a_j^{2\a} e^{-a_j^2/(4 s)}}{s^{1+\a}}  } \abs{f_2(t-s)}^{p'}~ ds\Big\}  \\
 & \quad \quad  \times  \Big\{\int_0^\infty\abs{\frac{ a_{j+1}^{2\a}e^{-a_{j+1}^2/(4 s)}-a_j^{2\a} e^{-a_j^2/(4 s)}}{s^{1+\a}}  }~ ds  \Big\}^{p'/p} \Big)^{1/p'}
\nonumber \\
&\le C \norm{v}_{l^p(\mathbb Z)}\left(\sum_{j=N_1}^{N_2} \int_0^\infty\abs{\frac{ a_{j+1}^{2\a}e^{-a_{j+1}^2/(4 s)}-a_j^{2\a} e^{-a_j^2/(4 s)}}{s^{1+\a}}  } \abs{f_2(t-s)}^{p'}~ ds   \right)^{1/p'}
\nonumber \\
&\le C \norm{v}_{l^p(\mathbb Z)}\left(\int_0^\infty\sum_{j=-\infty}^{+\infty} \abs{\frac{ a_{j+1}^{2\a}e^{-a_{j+1}^2/(4 s)}-a_j^{2\a} e^{-a_j^2/(4 s)}}{s^{1+\a}}  } \abs{f_2(t-s)}^{p'}~ ds  \right)^{1/p'}
\nonumber \\ \nonumber
&\le C \norm{v}_{l^p(\mathbb Z)}\left(\int_0^\infty \frac1{|s|} \abs{f_2(t-s)}^{p'}~ ds  \right)^{1/p'}\\
&\le  C \norm{v}_{l^p(\mathbb Z)}\int_{\real} \frac1{|t-s|} \abs{f_2(s)}^{p'}~ ds.
\end{align*}
For $s\in B\backslash B_{2r}$ and $t\in B_r$,  we have  $r\le |t-s|\le 2$.
 Then,  we get
\begin{align*}
\frac{1}{|B_r|} \int_{B_r} \abs{T^* f_2(t)}dt &\le    C\frac{1}{|B_r|} \int_{B_r} \left(\int_{\real} \frac1{|t-s|} \abs{f_2(s)}^{p'}~ ds  \right)^{1/p'}dt \\
&\le C\frac{\norm{f}_{L^\infty(\real)}}{|B_r|} \int_{B_r} \left(\int_{r\le |t-s| \le 2} \frac1{|t-s|} ~ ds  \right)^{1/p'}dt \\
 &\sim \Big(\log\frac{2}{r}\Big)^{1/p'}\norm{f}_{L^\infty(\real)}.
\end{align*}
Hence, $$\frac{1}{|B_r|} \int_{B_r} \abs{T^* f(t)}dt\le C\left(1+\Big(\log\frac{2}{r}\Big)^{1/p'}\right)\norm{f}_{L^\infty(\real)}  \le C\Big(\log\frac{2}{r}\Big)^{1/p'}\norm{f}_{L^\infty(\real)}.$$
For the case $p=1$ and $p=\infty$, the proof is similar and easier. Then we get the proof of $(a).$

For $(b),$ when $1< p<\infty,$ for any $0<\varepsilon<p-1$, let  \begin{equation*}\displaystyle f(t) = \sum_{k=-\infty}^{0} (-1)^{k} \chi_{(-a^{2k}, -a^{2k-1}]}(t) \,\,  \hbox{and  }\, \, a_j=a^{j}  ,
\end{equation*}
with $a>1$ being fixed later. Then, the support of $f$ is contained in $[-1, 0),$ and $\{a_j\}_{j\in \mathbb Z}$ is a $\rho$-lacunary sequence with $\rho=a>1.$ We observe that
\begin{eqnarray*}
\abs{\int_0^{+\infty} \frac{ e^{-1/(4u)}}{u^{\alpha}}  f(-u) \frac{du}{u}} \le
\int_0^{+\infty} \frac{ e^{-1/(4u)}}{u^{\alpha}} \frac{du}{u}=4^\alpha\Gamma(\alpha)< \infty.
\end{eqnarray*}
Hence $$\lim_{R\to {+\infty}} \int_R^{+\infty} \frac{ e^{-1/(4u)}}{u^{\alpha}}  f(-u)\frac{du}{u} =0\quad  \hbox{  and  } \quad \lim_{\varepsilon \to 0} \int_0^\varepsilon \frac{ e^{-1/(4u)}}{u^{\alpha}}  f(-u)\frac{du}{u} =0. $$
Also there exists a constant $C>0$ such that  $\displaystyle \lim_{a\to {+\infty}} \int_{a^{-1}}^1 \frac{ e^{-1/(4u)}}{u^{\alpha}}  f(-u)\frac{du}{u}  = \displaystyle \lim_{a\to {+\infty}} \int_{a^{-1}}^1\frac{ e^{-1/(4u)}}{u^{\alpha}} \frac{du}{u} =C.$  So we can choose $a>1$ big enough such that
 \begin{align*}
\int_{a^{-1}}^1 \frac{ e^{-1/(4u)}}{u^{\alpha}}  f(-u)\frac{du}{u} &= \int_{a^{-1}}^1\frac{ e^{-1/(4u)}}{u^{\alpha}}  \frac{du}{u}
 \ge 10 \left(\int_0^{1/a^2} \frac{ e^{-1/(4u)}}{u^{\alpha}} \frac{du}{u}+\int_{a-1}^{+\infty} \frac{ e^{-1/(4u)}}{u^{\alpha}} \frac{du}{u}\right)\nonumber\\
& >  10 \left( \abs{\int_0^{1/a^2} \frac{ e^{-1/(4u)}}{u^{\alpha}}f(-u) \frac{du}{u}}+\abs{\int_{a-1}^{+\infty} \frac{ e^{-1/(4u)}}{u^{\alpha}}f(-u) \frac{du}{u}}\right).
\end{align*}
Therefore, there exists a constant $C_1>0$ such that
\begin{equation}\label{equ:cons2}
\int_0^{+\infty} \frac{ e^{-1/(4u)}}{u^{\alpha}}  f(-u)\frac{du}{u}=C_1>0
\end{equation}
and
\begin{equation}\label{equ:cons33}
0<\int_0^{1/ a^2} \frac{ e^{-1/(4u)}}{u^{\alpha}}\frac{du}{u}+\int_{a-1}^{+\infty} \frac{ e^{-1/(4u)}}{u^{\alpha}}\frac{du}{u}\le {C_1\over 9}.
\end{equation}
On the other hand,
by the dominated convergence theorem, we have
\begin{equation*}
\lim_{h\to 0}\int_0^{+\infty}\frac{ e^{-1/(4u)}}{u^{\alpha}} f(h-u) \frac{du}{u}=\int_0^{+\infty}\frac{ e^{-1/(4u)}}{u^{\alpha}} f(-u) \frac{du}{u}  = C_1>0,
\end{equation*}
where $C_1$ is the constant appeared in \eqref{equ:cons2}.
So, there exists $0<\eta_0<1,$ such that, for $|h|<\eta_0,$
\begin{equation}\label{eq:bigC}
\int_0^{+\infty}\frac{ e^{-1/(4u)}}{u^{\alpha}} f(h-u) \frac{du}{u}\ge {1\over 2}\int_0^{+\infty}\frac{ e^{-1/(4u)}}{u^{\alpha}} f(-u) \frac{du}{u} = \frac{C_1}{2}.
\end{equation}

It can be checked that
$$f(a^{2j}t) = (-1)^j f(t)+(-1)^j\sum_{k=1}^{-j}(-1)^k \chi_{(-a^{2k}, -a^{2k-1}]}(t)$$
when $j\le 0.$ We will always assume $j\le 0$ in the following.
By changing variable,
\begin{align*}
\mathcal{P}_{a_j}^\alpha f(t) &=
\frac{1}{4^\alpha \Gamma(\alpha) }
\int_0^{+\infty} \frac{ e^{-1/(4u)}}{u^{\alpha}} f(t-a^{2j}u) \frac{du}{u}\\
& =
\frac{(-1)^j }{4^\alpha \Gamma(\alpha) }
\int_0^{+\infty} \frac{ e^{-1/(4u)}}{u^{\alpha}} \left\{ f\left({t\over a^{2j}}-u\right)+\sum_{k=1}^{-j}(-1)^k \chi_{(-a^{2k}, -a^{2k-1}]}\left({t\over a^{2j}}-u\right) \right\}\frac{du}{u}.
\end{align*}
Then
\begin{align}\label{equ:integ}
&{\mathcal{P}_{a_{j+1}}^\alpha f(t)-\mathcal{P}_{a_j}^\alpha f(t)}\nonumber\\
& =
\frac{ (-1)^{j+1}}{4^\alpha \Gamma(\alpha) }
\Big\{\, \int_0^{+\infty}\frac{ e^{-1/(4u)}}{u^{\alpha}} f\left({t\over {a^{2(j+1)}}}-u\right)  \frac{du}{u}+\int_0^{+\infty}\frac{ e^{-1/(4u)}}{u^{\alpha}} f\left({t\over {a^{2j}}}-u\right)\frac{du}{u} \\
&\quad +\int_0^{+\infty}\frac{ e^{-1/(4u)}}{u^{\alpha}}\sum_{k=1}^{-j-1}(-1)^k \chi_{(-a^{2k}, -a^{2k-1}]}\left({t\over a^{2j+2}}-u\right)\frac{du}{u}\nonumber\\
&\quad +\int_0^{+\infty}\frac{ e^{-1/(4u)}}{u^{\alpha}}\sum_{k=1}^{-j}(-1)^k \chi_{(-a^{2k}, -a^{2k-1}]}\left({t\over a^{2j}}-u\right) \frac{du}{u}\, \Big\}.\nonumber
\end{align}
{For given $\eta_0$ as above,} let $2r<1$ such that  $r< \eta_0^2$   and $r \sim  a^{2J_0}\eta_0$ for a certain negative integer $J_0$.  If $J_0\le j\le 0$, we have   $\displaystyle {r\over a^{2j}} <\eta_0$.
And,  for any $-r\le t\le r$  we have
$$\displaystyle -1\cdot \chi_{[a-1,+\infty)}(u)\le \sum_{k=1}^{-j-1}(-1)^k\chi_{(-a^{2k}, -a^{2k-1}]}\left({t\over a^{2j+2}}-u\right)\le \chi_{[a-1,+\infty)}(u)$$ and $$\displaystyle -1\cdot \chi_{[a-1,+\infty)}(u) \le \sum_{k=1}^{-j} (-1)^k\chi_{(-a^{2k}, -a^{2k-1}]}\left({t\over a^{2j}}-u\right)\le \chi_{[a-1,+\infty)}(u). $$
Hence, for the third and fourth integrals in \eqref{equ:integ}, by \eqref{equ:cons33} we have
\begin{align} \label{equ:large}
&\int_0^{+\infty}\frac{ e^{-1/(4u)}}{u^{\alpha}}\sum_{k=1}^{-j-1}(-1)^k \chi_{(-a^{2k}, -a^{2k-1}]}\left({t\over a^{2j+2}}-u\right)\frac{du}{u}+\nonumber\\
&\quad \quad \quad \quad\quad \quad\quad \quad\int_0^{+\infty}\frac{ e^{-1/(4u)}}{u^{\alpha}}\sum_{k=1}^{-j}(-1)^k \chi_{(-a^{2k}, -a^{2k-1}]}\left({t\over a^{2j}}-u\right) \frac{du}{u}\\
&\quad \quad \quad \quad\quad \quad\quad \quad\quad \quad \quad \quad\quad \quad \quad \quad \quad\quad \quad\quad \quad\ge (-2)\int_{a-1}^{+\infty}\frac{ e^{-1/(4u)}}{u^{\alpha}}\frac{du}{u}\ge -{2C_1\over 9}.\nonumber
\end{align}
So, for any $t\in [-r, r]$ and $J_0\le j\le 0$, combining \eqref{equ:integ},  \eqref{eq:bigC} and \eqref{equ:large},  we have
\begin{align*}
\abs{\mathcal{P}_{a_{j+1}}^\alpha f(t)-\mathcal{P}_{a_j}^\alpha f(t)} \ge C_\alpha \cdot\left( C_1-{2C_1\over 9}\right)=C\cdot C_1>0.
\end{align*}
We choose the sequence  $\{v_j\}_{j\in \mathbb Z} \in \ell^p(\mathbb Z)$ given by  $\displaystyle v_j=(-1)^{j+1}(-j)^{-{1\over p-\varepsilon}}$, then  for  $N=(J_0, 0),$   we have
\begin{multline*}
 \frac1{2r}\int_{[-r,r]} \abs{T^* f(t)} dt \ge \frac1{2r}\int_{[-r,r]} \abs{T_N^\alpha f(t)} dt \ge  \frac{1}{4^\alpha\Gamma(\alpha)} \frac 1{2r}\int_{[-r,r]}  \sum_{j =J_0}^{0} \left(C\cdot C_1\cdot (-j)^{-{1\over p-\varepsilon}}\right) dt\\ \ge C_{p,\varepsilon,\alpha}\cdot C_1\cdot (-J_0)^{1\over {(p-\varepsilon)'}}  \sim \left(\log \frac 2{r}\right)^{1\over {(p-\varepsilon)'}}.
 \end{multline*}

For $(c)$,  let $v_j=(-1)^{j+1}$,  $a_j=a^{j}$ with $a>1$  and  $0<\eta_0<1$ fixed in the proof of $(b)$.
Consider the same function $f$ as in $(b).$  Then, $\norm{v}_{l^\infty(\mathbb Z)}=1$ and $\norm{f}_{L^\infty(\mathbb R)}=1.$
By the same argument as in $(b)$, with $N=(J_0, 0)$ and $0<\alpha<1$,  we have
 $$\frac1{2r}\int_{[-r,r]} \abs{T^* f(t)} dt\ge \frac1{2r}\int_{[-r,r]} \abs{T_N^\alpha f(t)} dt \ge  \frac{1}{4^\alpha\Gamma(\alpha)} \frac 1{2r}\int_{[-r,r]}  \sum_{j = J_0}^{0} C_1 dt \ge \frac{ C_1}{4^\alpha\Gamma(\alpha)}\cdot (-J_0)  \sim \log \frac 2{r}.$$
\end{proof}
\medskip

\medskip

\noindent{\bf Acknowledgments.}   Zhang Chao is grateful to the Department of Mathematics at Universidad Aut\'onoma de Madrid for its hospitality during the period of this research.

\vspace{3em}


\begin{thebibliography}{10}




\bibitem{AFM}
H. Aimar, L. Forzani and F.J. Mart\'in-Reyes, On weighted inequalities for singular integrals. \textit{ Proc. Amer. Math. Soc.} 125  (1997),   2057--2064.





\bibitem{BLMMDT}
A.L. Bernardis, M. Lorente, F.J. Mart\'in-Reyes, M.T. Mart\'inez, A. de la Torre and J.L. Torrea,
\newblock Differential transforms in weighted spaces,
\newblock \textit{J. Fourier Anal. Appl.}
12 (2006), 83-103.

\bibitem{Bernardis} A. Bernardis, F.J. Mart\'in-Reyes, P.R. Stinga and J.L. Torrea,  Maximum principles, extension problem and inversion for nonlocal one-sided equations. \textit{J. Differential Equations}. 260 (2016) 6333-6362.



\bibitem{BCT}
J.J. Betancor, R. Crescimbeni and J.L. Torrea,
\newblock The $\rho$-variation of the heat semigroup in the Hermitian setting: behaviour in $L^\infty$.
\newblock \textit{Proceedings of the Edinburgh Mathematical Society.} 54 (2011), 569-585.


\bibitem{CaffarelliSil} L. Caffarelli and L. Silvestre,
{An extension problem related to the fractional Laplacian},
\textit{Comm. Partial Differential Equations.}
32 (2007), 1245--1260.

\bibitem{Duo}
J. Duoandikoetxea,  \textit{Fourier analysis}, Translated and revised from
the 1995 Spanish original by David Cruz-Uribe. Graduate Studies in
Mathematics, Volume 29, American Mathematical Society, Providence,
RI, 2001.





\bibitem{JR}
R.L. Jones and J. Rosenblatt,
\newblock Differential and ergodic transforms,
\newblock \textit{Math. Ann.}
323 (2002), 525-546.













\bibitem{MTX}
T.  Ma, J.L. Torrea and Q. Xu, Weighted variation inequalities for differential operators and singular integrals. \textit{J. Funct. Anal.}  268  (2015), 376--416.


\bibitem{Xu}
T. Mart\'{i}nez, J.L. Torrea and Q. Xu, Vector-valued
Littlewood--Paley--Stein theory for semigroups, \textit{Adv.  Math.}
{203} (2006),  430--475.




\bibitem{RubioRuTo}
J.L. Rubio de Francia, F.J. Ruiz and J.L. Torrea,
\newblock Calder\'on-Zygmund theory for operator-valued kernels,
\newblock \textit{Adv. in Math.} 62 (1986), 7-48.






\bibitem{Sawyer}
E. Sawyer, Weighted inequalities for the one-sided Hardy-Littlewood maximal functions. \textit{Trans. Amer. Math. Soc.}  {297}  (1986),  53--61.


\bibitem{StingaTorreaExten} P. R. Stinga and J.L. Torrea,
{Extension problem and Harnack's inequality for some fractional operators},
\textit{Comm. Partial Differential Equations.}
{35} (2010), 2092--2122.

\bibitem{StingaTorreaRegu}
P.R. Stinga and J.L. Torrea,
\newblock Regularity theory and extension problem for fractional nonlocal parabolic equations 	
and the master equation. \textit{SIAM J. Math. Anal.} {49}  (2017),  3893--3924.

%

\end{thebibliography}
\end{document}